\documentclass[11pt]{amsart}
\usepackage[divide={2.45cm,*,2.45cm}]{geometry}               
\usepackage{graphicx}
\usepackage{amssymb}
\usepackage{epstopdf}
\usepackage{amsmath}
\usepackage{amsfonts}
\usepackage{amssymb}
\usepackage{amsthm}

\title{The Transverse Entropy Functional and the Sasaki-Ricci flow}
\author{Tristan C. Collins}
\address{Department of Mathematics, Columbia University, New York, NY 10027}
\email{tcollins@math.columbia.edu}

\begin{document}
\theoremstyle{plain}
\newtheorem{Lemma}{Lemma}[section]
\newtheorem{prop}{Proposition}[section]
\newtheorem{Theorem}{Theorem}[section]
\newtheorem{corollary}{Corollary}[section]
\newtheorem{definition}{Definition}[section]
\newtheorem{example}{Example}
\theoremstyle{remark}
\newtheorem*{remark}{Remark}
\newtheorem{step}{Step}
\maketitle
\begin{abstract}
We introduce two new functionals, inspired by the work of Perelman, which are monotonic along the Sasaki-Ricci flow.  We relate their gradient flow, via diffeomorphisms preserving the foliated structure of the manifold, to the transverse Ricci flow.  Finally, when the basic first Chern class is positive, we employ these new functionals to prove a uniform $C^{0}$ bound for the transverse scalar curvature, and a uniform $C^{1}$ bound for the transverse Ricci potential along the Sasaki-Ricci flow. 
\end{abstract}
\section{Introduction and Statement of Results}
Recently, Sasakian geometry has seen significant interest due to its role in the AdS/CFT correspondence.  In particular, it is desirable to obtain a large class of Sasaki-Einstein manifolds.  A great deal of progress has been made toward answering the question of when Sasaki-Einstein metrics exist.  When the Sasaki manifold is regular, or quasi-regular (see \S 2 for definitions), much of the existence theory follows from the known results in K\"ahler geometry; we refer the reader to \cite{Sparks} and the references therein.  Until very recently, no examples of irregular Sasaki-Einstein manifolds were known to exist.  In fact, it was conjectured by Cheeger and Tian \cite{CheeTian} that any Sasaki-Einstein manifold must be quasi-regular.  However, this conjecture was disproved when irregular Sasaki-Einstein manifolds were discovered by Gauntlett, Martelli, Sparks and Waldram \cite{GMarSparW} coming from the physics surrounding the AdS/CFT correspondence.  This discovery necessitated a more robust approach to studying Sasaki-Einstein manifolds which did not require a Hausdorff topological structure on the leaf space.  Motivated by Calabi, another method for studying Sasaki-Einstein geometry via certain functionals related to the volume and scalar curvature was developed by Boyer, Galicki and Simanca in \cite{BoyGalSim} and Martelli, Sparks and Yau in \cite{MarSparYau}.  As in the K\"ahler case, when the basic first Chern class is positive there are known obstructions to the existence of Sasaki-Einstein metrics; see for example \cite{Futaki, GMarSparYau}.  It is reasonable to expect that a suitable generalization of the famous conjecture of Yau \cite{Yau} should hold in the Sasaki case; that is, existence of Sasaki-Einstein metrics should be equivalent to some geometric invariant theory notion of stability.  We refer the reader to \cite{PSstab} and the references therein for an introduction to this very active area of research. 

Another well known method for generating Einstein metrics on manifolds is the Ricci flow, introduced by Hamilton in \cite{Ham}, and extended to K\"ahler manifolds in \cite{Cao}.  Recently, this method was generalized to foliated Riemannian manifolds in \cite{Minoo}, and then applied to Sasaki manifolds in \cite{SmoWaZa}, where the authors were able to generalize the results of \cite{Cao}.  The Sasaki-Ricci flow may prove a very effective way of generating examples of Sasaki-Einstein manifolds, particularly once a robust existence theory is developed.  For example, in the K\"ahler case, when an Einstein metric is known to exist in the first Chern class, it was claimed by Perelman that  the K\"ahler-Ricci flow will converge to it (see \cite{PSstab, TianZhu} for a proof if there are no holomorphic vector fields).  This result was generalized in \cite{TianZhu} to include K\"ahler-Ricci solitions.  Moreover, there is a large body of work relating the various geometric notions of stability, as well as algebraic properties of the base manifold, to the convergence of the K\"ahler-Ricci flow; see for instance \cite{PSSW, PSesS,  PS, Valent}.   It would be desirable to extend these results to the Sasakian setting.  In this paper, we take a step in this direction.  In particular, we generalize some results of Perelman \cite{Pman, SesumTian} (see also \cite{BenChow, KleinerLott}), to Sasaki manifolds. 

Before outlining this paper, we remark that all of our results hold for irregular Sasaki manifolds.  In \S2 of this paper we define key notions of foliations, as well as Sasakian and transverse geometry.  We examine the relationship between the local transverse geometry and the global geometry.  A significant difficulty that arises in the irregular case is the absence of a transverse distance function which is a metric on the leaf space.  We define a transverse distance function which will be crucial to our developments, and investigate the properties of the ``balls" it defines.  In general, these sets can have pathological properties, but for sufficiently small radii we are able to identify them geometrically .

In \S 3 we define the Sasaki and transverse Ricci flows and state some of the known results. We also introduce special coordinates, developed in \cite{GKN}, which will aid us greatly in our subsequent developments.

In \S4 we introduce a new functional, which we call the transverse energy functional $\mathcal{F}^{T} : \mathfrak{Met}^{T}(S) \times C^{\infty}_{B}(S) \rightarrow \mathbb{R}$, given by
\begin{equation*}
\mathcal{F}^{T}(g,f) = \int_{S} (\text{R}^{T} + |\nabla f|^{2})e^{-f}d\mu.
\end{equation*}
We refer the reader to \S 2-4 for the relevant definitions.  We then prove:
\begin{Theorem}\label{trans F thm}
Let $(S,g_{0})$ be a Sasaki manifold, foliated by its Reeb field $\xi$.  Suppose that $g(t)$ is a solution of the transverse Ricci flow on $[0,T]$ with $g(0) = g_{0}$ and let $f_{T} \in C^{\infty}_{B}(S)$.  Then there is a solution $f(t)$ to the transverse backward heat equation
\begin{equation*}
\frac{\partial f}{\partial t} = -\Delta_{B}f + |\nabla f |^{2} - R^{T},
\end{equation*}
on $[0,T]$ with $f(T)=f_{T}$, and $f(t)$ basic for each $t \in [0,T]$.  Moreover, $\mathcal{F}^{T}(g(t), f(t))$ is monotonically increasing.
\end{Theorem}
A key feature of this new functional $\mathcal{F}^{T}$ is that we succeed in relating its gradient flow to the transverse Ricci flow via time dependent diffeomorphisms which preserve the foliation.  In particular, these diffeomorphisms fix the Reeb field and the integral curves of the Reeb field are geodesics with respect to the pulled-back metric.  However, the diffeomorphisms need not, in general, preserve the transverse holomorphic structure of the Sasaki manifold and this necessitates the consideration of a larger class of manifolds than the Sasaki manifolds.  We hope that this motivates the rather large number of definitions given in \S 2.  

In \S 5 we introduce another new functional, which is related to the functional $\mathcal{F}^{T}$ examined in \S 4.  We refer to this functional as the transverse entropy functional \newline $\mathcal{W}^{T}: \mathfrak{Met}^{T}(S) \times C^{\infty}_{B}(S) \times \mathbb{R}_{>0} \rightarrow \mathbb{R}$; it is given by
\begin{equation*}
\mathcal{W}^{T}(g,f,\tau) =(4\pi \tau)^{-n}\int_{S}\left( \tau(\text{R}^{T} + |\nabla f|^{2}) + (f-2n) \right) e^{-f}d\mu.
\end{equation*}
We then prove:
\begin{Theorem}\label{trans W thm}
Let $(S,g_{0})$ be a Sasaki manifold, foliated by its Reeb field $\xi$.  Suppose that $g(t)$ a solution of the transverse Ricci flow on $[0,T]$ with $g(0)=g_{0}$.  Let $\tau(t)$ be a positive function on $[0,T]$ with $\frac{d}{dt}\tau=-1$.  Let  $f_{T} \in C^{\infty}_{B}(S)$.  Then there is a solution $f(t)$ to the transverse backward heat equation
\begin{equation*}
\frac{\partial f}{\partial t} = -\Delta_{B}f + |\nabla f |^{2} - R^{T}+\frac{n}{\tau}
\end{equation*} 
on $[0,T]$ with $f(T)=f_{T}$, and $f(t)$ basic for each $t \in [0,T]$.  Moreover, $\mathcal{W}^{T}(g(t), f(t), \tau(t))$ is monotonically increasing. 
\end{Theorem}

Finally, we apply these results to extend Perelman's results on the K\"ahler-Ricci flow \cite{SesumTian} to the Sasaki setting.  Namely, we prove;

\begin{Theorem}\label{Pman thm}
Let $g^{T}(t)$ be a Sasaki-Ricci flow on a Sasaki manifold $(S,\xi)$ of real dimension $2n+1$, and transverse complex dimension n, with $c^{1}_{B}(S)>0$.  Let $u \in C^{\infty}_{B}(S)$ be the transverse Ricci potential.  Then there exists a uniform constant $C$, depending only on the initial metric $g(0)$ so that
\begin{equation}
|R^{T}(g(t))| + |u|_{C^{1}} + diam^{T}(S,g(t)) < C
\end{equation}
where $diam^{T}(S,g(t)) = \sup_{x,y\in S} d^{T}(x,y)$.
\end{Theorem} 

Again, we refer the reader to \S 2 and \S 3 for the relevant definitions.  The proof of this result relies heavily on our developments in \S 2-5.  We observe that Theorem~\ref{Pman thm} is evidence for the conjecture of Smoczyk, Wang, and Zhang \cite{SmoWaZa} that the Sasaki-Ricci flow converges, in some suitable sense, to a Sasaki-Ricci soliton, a notion introduced in \cite{Futaki}.

\begin{remark}
Very shortly after posting this paper on the arXiv, a paper by Weiyong He titled \emph{The Sasaki-Ricci flow and Compact Sasakian Manifolds of Positive Transverse Holomorphic Bisectional Curvature} (arXiv:1103.5807) appeared.  There is overlap in our results and methods, though He confronts the difficulties which arise in the irregular case differently.  He also considers the Sasaki-Ricci flow with positive, transverse bisectional curvature, while we obtain a non-collapsing result which holds for irregular Sasaki manifolds.
\end{remark}

{\bf Acknowledgements:}
I would like to thank my advisor Professor D.H. Phong for suggesting this problem, as well as his encouragement and advice in the writing of this paper.  I would also like to thank Professor Valentino Tosatti for many helpful conversations and for patiently reading a rather rough first draft.  The majority of the research and writing of this paper was carried out during a visit to the Centre International Rencontres Math\'ematiques, whose support I would like to graciously acknowledge.  I would like to thank the referee for many helpful suggestions.

\section{Background}
In this section we introduce the key concepts in the theory of foliations and Sasaki manifolds needed for our work.  We refer the reader to \cite{BoyGal1, Molino, Tond} for a thorough introduction to the subject of foliated manifolds, and transverse Riemannian geometry.  We refer the reader to \cite{BoyGal1, Sparks} for an introduction to Sasakian geometry.

\subsection{Riemannian Flows and Riemannian Foliations}
Let $S$ be a compact manifold of dimension $2n+1$, and $\xi$ a nowhere vanishing vector field on $S$.  The integral curves of $\xi$ generate a foliation of $S$ by one dimensional submanifolds.  Assume now that $(S,\xi,g)$ is a foliated Riemannian manifold, and denote by $L_{\xi}$ the subbundle of $TS$ generated by $\xi$.
\begin{definition}
A function $f \in C^{\infty}(S)$ is said to be \emph{basic} if $\mathcal{L}_{\xi}f=0$.  The set of smooth basic functions will be denoted by $C^{\infty}_{B}(S)$.
\end{definition}
\begin{definition}
A vector field $X$ on $S$ is said to be \emph{foliate} with respect to the foliation induced by $\xi$ if $[X,\xi] \in L_{\xi}$.
\end{definition}
\begin{definition}\label{bundle-like}
A Riemannian metric g is said to be \emph{bundle-like} with respect to the foliation induced by $\xi$ if for any open set $U$ and foliate vector fields $X$, $Y$ on $U$ perpendicular to $L_{\xi}$, the function $g(X,Y)$ is basic.
\end{definition}
The foliation induced by $\xi$ on $S$ is said to be a Riemannian foliation if $S$ admits a bundle-like metric $g$.  Throughout this paper we will remain firmly in the category of Riemannian foliations.  Riemannian foliations admit global transverse Riemannian structures which reflect the geometry of the local Riemannian quotients.  That is, given a foliated Riemannian manifold $(S,\xi,g)$ we obtain an exact sequence
\begin{equation}\label{exact seq}
0\rightarrow L_{\xi} \rightarrow TS \xrightarrow{p} Q \rightarrow 0
\end{equation}
where $Q=TS/L_{\xi}$.  The metric $g$ defines an orthogonal splitting of this sequence $Q\xrightarrow{\sigma}TS$, so that we may write
\begin{equation*}
TS = L_{\xi}^{\perp} \oplus L_{\xi}.
\end{equation*}
By identifying $L_{\xi}^{\perp}\cong Q$, we get a metric $g^{T}$ on $Q$ by restricting $g$ to $L_{\xi}^{\perp}$.  In general, however, this metric will \emph{not} yield a metric on the local Riemannan quotient. In fact, if $U$ is a small open set, and $\pi:U\rightarrow \bar{U}$ is the quotient map to the local Riemannian quotient, then $g^{T}=\pi^{*}\tilde{g}$ for some metric $\tilde{g}$  on $\bar{U}$ if and only if $g$ is bundle-like.  In the case that $(S,\xi,g)$ is a Riemannian foliation, we see the above procedure produces a global metric $g^{T}$ on $Q$ which is given locally by the pull-back of a metric from the local Riemannian quotient.
If $(S,\xi,g)$ is a Riemannian foliation such that $g(\xi,\xi)=1$, and the integral curves of $\xi$ are geodesics, then the foliation is said to be a \emph{Riemannian flow}.  A key feature of manifolds foliated by a Riemannian flow is that for any point $p \in S$ we can construct local coordinates in a neighbourhood of $p$ which are simultaneously foliated, and Riemann normal coordinates.  That is, we can find Riemann normal coordinates $\{x,y_{1},\dots,y_{2n}\}$ on a neighbourhood $U$ of $p$, such that $\frac{\partial}{\partial x} =\xi$ on $U$.

\subsection{Sasaki Manifolds and Contact structures}
A Sasakian manifold of dimension $2n+1$ is a Riemannian manifold $(S^{2n+1},g)$ with the property that its metric cone $(C(S) = \mathbb{R}_{>0}\times S, \overline{g} = dr^2 + r^{2}g)$ is K\"ahler.  A Sasakian manifold inherits a number of key properties from its K\"ahler cone.  In particular, an important role is played by the Reeb vector field.  
\begin{definition}
The Reeb vector field is $\xi = J(r\partial_{r})$, where $J$ denotes the integrable complex structure on $C(S)$.  
\end{definition}
The restriction of $\xi$ to the slice $\{r=1\}$  is a unit length Killing vector field, and its orbits define a one-dimensional foliation of S by geodesics called the Reeb foliation.  There is a dual one-form $\eta$ defined by $\eta = id^{c}\log r = i (\overline{\partial} -\partial)\log r$, which has the properties
\begin{equation*}
\eta(\xi)=1, \text{  } \iota_{\xi}d\eta=0.
\end{equation*}
Moreover, we have that for all vector fields X
\begin{equation*}
\eta(X) = \frac{1}{r^{2}}\overline{g}(\xi,X).
\end{equation*}
The K\"ahler form on $C(S)$ is then given by $\omega = \frac{1}{2}d(r^{2}\eta)$.  The 1-form $\eta$ restricts to a 1-form $\eta |_{S}$ on $S \subset C(S)$.  Using the fact the $\mathcal{L}_{r\partial_{r}}\eta =0$ one shows that in fact $\eta = p^{*}(\eta |_{S})$ where $p$ is the canonical projection from $C(S)$ to $S$.  By abuse of notation, we do not distinguish between $\eta$ and its restriction to $S$.  Since the K\"{a}hler 2-form $\omega$ is non-vanishing, it follows that the top degree form $\eta \wedge (d\eta)^{n}$ is non-vanishing on $S$, and hence $(S,\eta)$ is a contact manifold.

Let $L_{\xi}$ be the line bundle spanned by the non-vanishing vector field $\xi$.  The contact subbundle $D\subset TS$ is defined as $D= \ker \eta$.  We then have the exact sequence~(\ref{exact seq}).  The Sasakian metric $g$ gives an orthogonal splitting of this sequence so that we identify $Q \cong D$, and
\begin{equation*}
TS = D\oplus L_{\xi}
\end{equation*}
Define a section $\Phi \in End(TS)$ via the equation $\Phi(X) = \nabla_{X} \xi$.  One can then check that $\Phi |_{D} = J |_{D}$ and $\Phi |_{L_{\xi}} = 0$, and that
\begin{equation*}
\Phi^{2} = -1 + \eta \otimes \xi \text{  } \text{, and   } \text{  } g(\Phi(X), \Phi(Y)) = g(X,Y)-\eta(X)\eta(Y)
\end{equation*}
for any vector fields $X$ and $Y$ on S.  The second equation says that $g |_{D}$ is a Hermitian metric on D.  Since
\begin{equation}\label{metric}
g(X,Y) = \frac{1}{2} d\eta (X, \Phi (Y)) + \eta(X)\eta(Y),
\end{equation}
 we see that $\frac{1}{2} d\eta |_{D}$ is the fundamental 2-form associated to $g |_{D}$.  The triple $(D, \Phi |_{D}, d\eta)$ gives S a transverse K\"{a}hler structure.  We take this up in the following.

\subsection{Transverse K\"{a}hler structures and the Reeb foliation}

The Reeb foliation inherits a transverse holomorphic structure and K\"{a}hler metric in the following (explicit) way.  The leaf space of the foliation induced by the Reeb field can clearly be identified with the leaf space of the \emph{holomorphic} vector field $\xi -iJ(\xi)$ on the cone $C(S)$.  Here, $J$ denotes the integrable complex structure on the cone.  By using holomorphic, foliated coordinates on $C(S)$, we may introduce a foliation chart  $\{U_{\alpha}\}$ on $S$, where each $U_{\alpha}$ is of the form $U_{\alpha} = I \times V_{\alpha}$ with $I\subset \mathbb{R}$ and open interval, and $V_{\alpha} \subset \mathbb{C}^{n}$.  We can find coordinates $(x,z_{1}, \dots, z_{n})$ on $U_{\alpha}$, where $\xi=\partial_{x}$, and $z_{1},\dots,z_{n}$ are complex coordinates of $V_{\alpha}$.  The fact that the cone is complex implies that the transition functions between the $V_{\alpha}$ are holomorphic.  More precisely, if $(y, w_{1}, \dots, w_{n})$ are similarly defined  coordinates on $U_{\beta}$ with $U_{\alpha} \cap U_{\beta} \neq \emptyset$, then
\begin{equation*}
\frac{\partial z_{i}}{\partial \overline{w}_{j}} = 0 \text{ ,  } \text{  } \text{  } \frac{\partial z_{i}}{\partial y} = 0.
\end{equation*}
Recall that the subbundle $D$ is equipped with the almost complex structure $J |_{D}$, so that on $D\otimes \mathbb{C}$ we may define the $\pm i$ eigenspaces of $J |_{D}$ as the $(1,0)$ and $(0,1)$ vectors respectively.  Then, in the above foliation chart, $(D \otimes \mathbb{C})^{(1,0)}$ is spanned by $\partial_{z_{i}} -\eta(\partial_{z_{i}})\xi$.
Since $\xi$ is a killing vector field it follows that $g|_{D}$ gives a well-defined Hermitian metric $g_{\alpha}^{T}$ on the patch $V_{\alpha}$.  Moreover, ~(\ref{metric}) implies that
\begin{equation*}
d\eta(\partial_{z_{i}} -\eta(\partial_{z_{i}})\xi, \partial_{\overline{z}_{j}} -\eta(\partial_{\overline{z}_{j}})\xi)=d\eta(\partial_{z_{i}},\partial_{\overline{z}_{j}}).
\end{equation*}
Thus, the fundamental 2-form $\omega_{\alpha}^{T}$ for the Hermitian metric $g_{\alpha}^{T}$ in the patch $V_{\alpha}$ is obtained by restricting $\frac{1}{2}d\eta$ to a fibre $\{x=constant\}$.  It follows that $\omega_{\alpha}^{T}$ is closed, and the transverse metric $g_{\alpha}^{T}$ is K\"{a}hler.  Note that in the chart $U_{\alpha}$ we may write
\begin{equation*}
\eta = dx + i \sum_{i=1}^{n}\partial_{z_{i}}K_{\alpha}dz_{i} -i\sum_{i=1}^{n}\partial_{\overline{z}_{i}}K_{\alpha}d\overline{z}_{i}
\end{equation*}
where $K_{\alpha}$ is a function of $U_{\alpha}$ with $\partial_{x}K_{\alpha}=0$.  We may identify the collection of transverse metrics $\{g_{\alpha}^{T}\}$ with the global tensor field on S given by
\begin{equation*}
g^{T}(X,Y) = \frac{1}{2}d\eta(X, \Phi(Y))
\end{equation*}
and the  transverse K\"{a}hler form $\omega^{T}$ is similarly defined globally as $\frac{1}{2}d\eta$.

Sasaki manifolds fall in to three categories based on the orbits of the Reeb field.  If the orbits of the Reeb field are all closed, then the $\xi$ generates a locally free, isometric $U(1)$ action on $(S,g)$.  If the $U(1)$ action is free, then $(S,g)$ is said to be \emph{regular} and the quotient manifold $S/U(1)$ is K\"ahler.  If the action is not free, then $(S,g)$ is said to be \emph{quasi-regular}, and the quotient manifold is a K\"ahler orbifold.  If the orbits of $\xi$ are not closed, then the Sasakian manifold $(S,g)$ is said to be irregular.   Carri\`ere \cite{Carr} has shown that the leaf closures are diffeomorphic to tori, and that the Reeb flow is conjugate via this diffeomorphism to a linear flow on the torus.  The local flow of the Reeb field defines  a commutative subgroup of the isometry group of $(S,g)$ whose closure is a torus $\mathfrak{T}$.  The dimension of $\mathfrak{T}$ is called the rank of $(S,g)$, denoted $rk(S,g)$, and is an invariant of the Pfaffian structure $(S,\xi)$.  If $(S,g)$ has dimension 2n+1, then one can show (see \cite{BoyGal}) that $1\leq rk(S,g) \leq n+1$.  In particular, for any $p \in S$, $\overline{orb_{\xi}p}$ is an imbedded torus of dimension less or equal $n+1$.  Here $orb_{\xi}p$ denotes the orbit of $p$ under the action generated by $\xi$.

\subsection{Transverse Geometry and Basic Cohomology }
We return now to the subject of transverse Riemannian geometry.  Suppose that $(S,\xi,g)$ is a manifold, foliated by the local flow induced by $\xi$, and $g$ is a bundle-like metric.  Then the quotient bundle $Q$ defined by~(\ref{exact seq}) is endowed with a transverse metric $g^{T}$.  Moreover, there is a unique, torsion-free connection on $Q$ which is compatible with the metric $g^{T}$.  This connection is defined by
\begin{displaymath}
\nabla^{T}_{X}V = \left\{ \begin{array}{ll}
	\left(\nabla_{X}\sigma(V)\right)^{p},& \text{if } X \text{ is a section of D}\\
	\left[\xi, \sigma(V)\right]^{p}, &\text{if } X=\xi
\end{array}
\right.
\end{displaymath}
where $\nabla$ is the Levi-Civita connection on $(S,g)$, $\sigma$ is the splitting map induced by $g$, and $p$ is the projection from $TS$ to the quotient $Q$.  This connection is called the \emph{transverse Levi-Civita connection} as it satisfies
\begin{equation*}
\nabla_{X}^{T}Y - \nabla^{T}_{Y}X - [X,Y]^{p}=0,
\end{equation*}
\begin{equation*}
Xg^{T}(V,W) = g^{T}(\nabla_{X}^{T}V,W)+ g^{T}(V, \nabla^{T}_{X}W).
\end{equation*}
In this way it is easy to see that the transverse Levi-Civita connection is the pullback of the Levi-Civita connection on the local Riemannian quotient.  We can then define the transverse curvature operator by
\begin{equation*}
Rm^{T}(X,Y) = \nabla_{X}^{T}\nabla_{Y}^{T} -\nabla_{Y}^{T}\nabla_{X}^{T}-\nabla_{[X,Y]}^{T},
\end{equation*}
and one can similarily define the transverse Ricci curvature, and the transverse scalar curvature.  There is also a notion of transverse cohomology, which we introduce now.
\begin{definition}
A $p$-form $\alpha$ on $(S,\xi)$ is called basic if $\iota_{\xi}\alpha=0$, and $\mathcal{L}_{\xi}\alpha =0$.
\end{definition}
On $(S,\xi)$ we let $\Lambda_{B}^{p}$ be the sheaf of basic $p$-forms, and $\Omega^{p}_{B} = \Gamma(S, \Lambda_{B}^{p})$ the global sections.  It is clear that the de Rham differential $d$ preserves basic forms, and hence restricts to a well defined operator $d_{B}:\Lambda^{p}_{B}\rightarrow \Lambda^{p+1}_{B}$.  We thus get a complex
\begin{equation*}
0\rightarrow C^{\infty}_{B}(S)\rightarrow \Omega^{1}_{B}\xrightarrow{d_{B}}\cdots \xrightarrow{d_{B}} \Omega^{2n}_{B} \xrightarrow{d_{B}} 0
\end{equation*}
whose cohomology groups, denoted by $H^{p}_{B}(S)$, are the basic de Rham cohomology groups.  Moreover, we can define the basic Laplacian on forms by
\begin{equation*}
-\Delta_{B} = d_{B}d_{B}^{\dagger} + d_{B}^{\dagger}d_{B}.
\end{equation*}
In general, the basic Laplacian \emph{does not} agree with the restriction of the de Rham Laplacian to basic forms \cite{KamTond}.  However, the basic Laplacian on functions \emph{does} agree with the restriction of the metric Laplacian to basic functions, and one can check that
\begin{equation*}
\Delta_{B}=\Delta |_{C^{\infty}_{B}(S)} = (g^{T})^{ij}\nabla_{i}\nabla_{j}
\end{equation*}
If $(S,g)$ is Sasakian, then the transverse complex structure $\Phi$ allows us to decompose
\begin{equation*}
\Lambda^{r}_{B}\otimes \mathbb{C} = \mathop{\oplus}_{p+q=r} \Lambda^{p,q}_{B}.
\end{equation*}
We can then decompose $d_{B} = \partial_{B} + \overline{\partial}_{B}$, where
\begin{equation*}
 \partial_{B}: \Lambda^{p,q}_{B}\rightarrow \Lambda^{p+1,q}_{B}, \text{ and }\text{   } \overline{\partial}_{B}:\Lambda^{p,q}_{B}\rightarrow \Lambda^{p,q+1}_{B}.
 \end{equation*}
Consider the form $\rho^{T}= Ric^{T}(\Phi \cdot, \cdot)$, which is called the \emph{transverse Ricci form}.  In analogy with the K\"ahler case, one can check that
\begin{equation*}
\rho^{T} = -\sqrt{-1}\partial_{B} \overline{\partial}_{B} \log \det (g^{T})
\end{equation*}
and hence $\rho^{T}$ defines a basic cohomology class, $[\rho^{T}]_{B}$, which is called the \emph{basic first Chern class}, and is independent of the transverse metric.  A transverse metric $g^{T}$ is called a \emph{transverse Einstein metric} if it satisfies
\begin{equation*}
Ric^{T}=cg^{T}
\end{equation*}
for some constant $c$.  In particular, we observe that it is necessary that the basic first Chern class be signed.
\begin{definition}
For $x,y \in S$ we define the transverse distance function $d^{T}: S\times S \rightarrow \mathbb{R}$ by
\begin{equation}
d^{T}(x,y) = \inf_{\{p\in orb_{\xi}x, q \in orb_{\xi}y\}} dist(p,q)
\end{equation}
\end{definition}
An important property of the transverse distance function is the following elementary lemma.
\begin{Lemma}
Fix $ z \in S$, and define a function $h:S\rightarrow \mathbb {R}$ by $h(y) = d^{T}(z,y)$.  Then $h$ is basic, and Lipschitz, with $|h(x)-h(y)|\leq dist(x,y)$.  In particular, $h$ is $a.e.$ differentiable, with $|\nabla h| \leq 1$.
\end{Lemma}
Note that in general, we can not expect any regularity better than Lipschitz.  In fact, if the Reeb field is irregular, then the transverse distance function fails to be a metric on the leaf space, and in general, can fail to be differentiable.  To see the kind of pathology that can occur in this case, consider a solid torus with a central $S^{1}$, foliated by tori, which are the orbits of irrational curves.  We regard the solid torus as a subset of $\mathbb{R}^{3}$ with the induced metric. In this case, it is clear that the distance function fails to be $C^{1}$, and that every point on the same torus is distance 0 from every other point, including points \emph{not} on the same orbit.  In the regular and quasi-regular cases, the transverse distance function defines a metric on the leaf space.  The analogue of a closed geodesic ball is;
\begin{definition}
The closed, transverse tube of radius $r$ about $x \in S$ is the set
\begin{equation*}
T(x,r) = \{y \in S | d^{T}(x,y) \leq r\}.
\end{equation*}
\end{definition}
Again, it is instructive to keep in mind the pathology that can occur when the Reeb field has non-closed orbits.  If $p$ is a point in $S$ with non-closed orbit, then its closure defines an imbedded torus of dimension $q\leq n+1$.  The transverse tube $T(p,r)$, is then the union of transverse tubes about {\it every} point in the torus.  We are the led to consider the geometry of such sets.  To make this plain, let $(M, g)$ be a Riemannian manifold and $P$ be a topologically imbedded submanifold, possibly with boundary.
\begin{definition}
The closed geodesic tube around $P$ of radius $r$ is the set
\begin{equation*}
T_{geo}(P,r) = \left\{ m \in S | \text{  }\exists \text{ a geodesic }\gamma \text{ of length }L(\gamma)\leq r \text{ from }m \text{ meeting } P \text{ orthogonally } \right\}.
\end{equation*}
Observe that when $r\leq inj(S)$, then we can write
\begin{equation*}
T_{geo}(P,r) = \bigcup_{p\in P } \left\{exp_{p}(v) | v\in T_{p}P^{\perp} \text{ and } \|v\| \leq r \right\}
\end{equation*}
\end{definition}

It is a perhaps surprising fact that, for sufficiently small radius, transverse tubes and geodesic tubes are equivalent. 

\begin{prop}\label{transverse v geo prop}
Let $p\in S$, and $P = \overline{ orb_{\xi}p}$ be the torus defined by $p$.  If $r \leq inj(S)$, then
\begin{equation*}
T(p,r) =  T_{geo}(P,r)
\end{equation*}
\end{prop}

\begin{proof}
First we prove the easy inclusion.  Assume that $x \in T_{geo}(P,r)$.  Then there is a $y \in P$, and $v \in T_{y}P^{\perp}$ with $ \|v \| \leq r$ such that $ x = exp_{y}(v) $.  Fix $\epsilon >0$, and choose $y_{\epsilon} \in orb_{\xi}p$ such that $d(y_{\epsilon}, y) < \epsilon$.  Then,
\begin{equation*}
d^{T}(p,x) = d^{T}(y_{\epsilon}, x) \leq d(y_{\epsilon}, x) \leq r+\epsilon.
\end{equation*}
Since this holds for every positive $\epsilon$, we get $x\in T(p,r)$.

To prove the reverse inclusion assume that $x \in T(p,r)$. By definition of $d^{T}$, we can find a $y \in P$ and $z \in \overline{ orb_{\xi}x}$ such that $d^{T}(z,y) \leq r $ and $d^{T}(x,y)= d(z,y)$.  Since $r\leq inj(S)$, there is a unique geodesic $\gamma$ joining $y$ to $z$ which has $L(\gamma) = d(z,y)$.  By an exercise in differential geometry $\gamma'(0) \in  T_{y}P^{\perp}$.  Thus, $z \in T_{geo}(P,r)$.  If $orb_{\xi}x$ is closed, then we are done, so assume otherwise.  We need a lemma.

\begin{Lemma}\label{points not closed}
If $z \in \overline{orb_{\xi}x}$, then $x\in \overline{orb_{\xi}z}$.  In particular, $orb_{\xi}z$ is not closed.
\end{Lemma}
\begin{proof}
This follows immediately from the fact that the Reeb field generates a Riemannian flow.  First, it is clear that if $z \in \overline{orb_{\xi}x}$, then $orb_{\xi}z \subset \overline{orb_{\xi}x}$, as the Reeb field is Killing.  Now, given a sequence $\{x_{n}\}\subset orb_{\xi}x$ converging to $z$, we generate a sequence $\{z_{n}\} \subset orb_{\xi}z$ converging to $x$ by composing with the diffeomorphism $\phi_{n}$ generated by $\xi$ so that $\phi_{n}(x_{n})=x$.  We then define $z_{n} = \phi_{n}(z)$.  That $z_{n}$ converges to $x$ follows immediately from the fact that $\phi_{n}$ is an isometry.
\end{proof}

Now, since $z \in T_{geo}(P,r)$, it is clear that $orb_{\xi}z \subset T_{geo}(P,r)$, as $\xi$ is Killing, and thus  $\overline{orb_{\xi}z }\subset T_{geo}(P,r)$, as $T_{geo}(P,r)$ is closed.  By Lemma~\ref{points not closed}, we see that $x\in T_{geo}(P,r)$.
\end{proof}
It is to our advantage that a great deal of work has been done to understand the relationship between the geometry of a tubular set around a submanifold $P$, and the ambient manifold $M$.  In particular, we have the following asymptotic expansion for the volume of a geodesic tube:
\begin{Theorem}[\cite{Gray} p. 203, Theorem 9.23]\label{Gray Thm}
Let $M$ be a manifold of dimension $n$, and $P\subset M$ a submanifold of dimension $q$. Let $\mathbb{V}^{M}_{P}(r)$ be the volume of a geodesic tube about $P$.  Then the following asymptotic expansion for the volume holds;
\begin{equation}
\mathbb{V}^{M}_{P}(r) = \frac{(\pi r^{2})^{\frac{1}{2}(n-q)}}{(\frac{1}{2}(n-q)!)}Vol(P)\left(1 + O(r^{2})\right).\end{equation}
\end{Theorem}
It is of note that in \cite{Gray} the first and second coefficients in the expansion are computed.  However, we will not need them.

\section{The Sasaki-Ricci flow, and the Transverse Ricci flow}
We now define the Sasaki-Ricci flow, and the transverse Ricci flow.  Given a Riemannian foliation $(S,\xi, g)$ we define the Transverse Ricci flow by
\begin{equation}\label{T-Ricci}
\frac{\partial g^T}{\partial t} = -Ric^{T}.
\end{equation}
The short time existence for the flow~(\ref{T-Ricci}) was established in \cite{Minoo}.  On a Sasakian manifold, we can exploit the transverse K\"ahler structure to introduce another flow for the transverse metric as follows;  fix an initial Sasaki metric $g_{o}=d\eta_{o}$ such that $\kappa[d\eta_{o}]_{B} = c^{1}_{B} = [Ric^{T}]_{B}$.  Using the transverse $\partial \bar{\partial}$-lemma of \cite{ElKac}, there is a basic function $F:S\rightarrow \mathbb{R}$ such that
\begin{equation*}
Ric^{T} = \kappa d\eta_{o} + d_{B}d_{B}^{c}F.
\end{equation*}
We then define the Sasaki-Ricci flow by
\begin{equation}\label{SRF}
\frac{\partial g^T}{\partial t} = -Ric_{g(t)}^{T} + \kappa g^{T}(t),
\end{equation}
which can be expressed locally as a parabolic Monge-Amp\`ere equation on transverse K\"ahler potentials $\phi$ via
\begin{equation}\label{parabolic MA}
\frac{\partial \phi}{\partial t} = \log \det(g_{\bar{k}l}^{T}+\partial_{l}\partial_{\bar{k}}\phi)- \log \det(g_{\bar{k}l}^{T})+\kappa \phi -F.
\end{equation}

It was proved in \cite{SmoWaZa} that this flow is well-posed.  Moreover, they showed that this flow preserves the Sasakian structure of $S$, the solution $\phi$ exists for all time, remains basic, and converges exponentially if $c^{1}_{B} \leq 0$.  Observe that one can pass from a solution to~(\ref{T-Ricci}) and a solution to the Sasaki-Ricci flow~(\ref{SRF}) via the usual method of dilating the metric and scaling time.  In particular, the results of \cite{SmoWaZa} imply that if the initial metric $g_{o}$ is Sasakian, then the solution to~(\ref{T-Ricci}) exists on $[0,\kappa^{-1})$ and remains Sasakian.  It will be important for us that these two flows are interchangeable.

We make a brief comment on volume forms.  For a contact manifold $(S,\eta)$, the top form $\eta\wedge (d\eta)^{n}$ defines a volume form on $S$.  If, on the other hand, $(S,\xi,g)$ is a foliated Riemannian manifold with the foliation induced by $\xi$, bundle-like metric $g$ and transverse metric $g^{T}$,  then we can use the standard Riemannian volume form on $S$.  Now, if $(S,g)$ is a Sasaki manifold, then $S$ is both a contact manifold, and a foliated Riemannian manifold, with metric $g=\eta\otimes\eta + d\eta(\cdot, \Phi\cdot)$, and induced transverse metric $g^{T}(X,Y) = d\eta(X,\Phi Y)$; it is an easy exercise to check that the volume form $\eta\wedge (d\eta)^{n}$ agrees with the Riemannian volume form.  Suppose now that the transverse metric $g^{T}$, is evolving by the Sasaki-Ricci flow.  Then the evolved contact structure is given by
\begin{equation*}
\eta_{t}= \eta_{o} +d_{B}^{c}\phi(t)
\end{equation*}
for a basic function $\phi(t)$. The evolved volume form is given by
\begin{equation*}
\eta_{t}\wedge (d\eta_{t})^{n} = \left(\eta_{o} +d_{B}^{c}\phi(t)\right)\wedge (d\eta_{t})^{n}= \eta_{o} \wedge (d\eta_{t})^{n}+ d_{B}^{c}\phi(t)\wedge (d\eta_{t})^{n}
\end{equation*}
By observing that $(d\eta_{t})^{n}\wedge d_{B}^{c}\phi(t)$ is a basic form of degree $2n+1$, we see immediately that
\begin{equation*}
\eta_{t}\wedge(d\eta_{t})^{n} = \eta_{o}\wedge(d\eta_{t})^{n} =  \det(g_{\bar{k}l}^{T}+\partial_{l}\partial_{\bar{k}}\phi).
\end{equation*}
See, for example, Lemma 5.2 of \cite{SmoWaZa}.  In particular, the canonical volume forms arising from the contact, and foliated Riemannian structures on a Sasaki manifold agree with the local transverse K\"ahler volume form.  The following lemma is an easy consequence of the evolution equation for the Sasaki-Ricci flow.

\begin{Lemma}
The volume of $S$ is preserved by the Sasaki-Ricci flow.
\end{Lemma}
We now make a brief digression on coordinates.  Let $(S,g)$ be a Sasaki manifold, $p\in S$ a point.  One can choose local coordinates $(x, z^{1},\dots,z^{n})$ on a small neighbourhood $p\in U$ such that
\begin{enumerate}
\item[\textbullet] $\xi = \frac{\partial}{\partial x}$
\item[\textbullet] $ \eta = dx + \sqrt{-1}\sum_{j=1}^{n}h_{j}dz^{j} - \sqrt{-1}\sum_{j=1}^{n}h_{\bar{j}}d\bar{z}^{j}$
\item[\textbullet] $ \Phi = \sqrt{-1} \left\{ \sum_{j=1}^{n}\left( \frac{\partial}{\partial z^{j}} -\sqrt{-1}h_{j}\frac{\partial}{\partial x}\right)\otimes dz^{j} -  \sum_{j=1}^{n}\left( \frac{\partial}{\partial \bar{z}^{j}} +\sqrt{-1}h_{\bar{j}}\frac{\partial}{\partial x}\right)\otimes d\bar{z}^{j}\right\} $
\item[\textbullet] $g= \eta \otimes \eta + 2\sum_{j,l=1}^{n}h_{j\bar{l}}dz^{j}d\bar{z}^{l}$
\end{enumerate}
where $h: U\rightarrow \mathbb{R}$ is a local, basic function (ie.  $\frac{\partial}{\partial x} h=0$), and we have used $h_{j} = \frac{\partial }{\partial z^{j}}h$, and $h_{j\bar{l}} =  \frac{\partial^{2} }{\partial z^{j}\partial \bar{z}^{l}}h$.  See, for example, \cite{GKN}.  We can additionally assume that in these coordinates $h_{j}(p)$=0.  An important observation is that the transverse Christoffel symbols with mixed barred and unbarred indices are identically zero.  Moreover, the pure barred and unbarred Christoffel symbols are given by the familiar formula for K\"ahler geometry
\begin{equation*}
\Gamma_{ij}^{k} = (g^{T})^{k\bar{p}}\partial_{i}g^{T}_{\bar{p} j}.
\end{equation*}
Throughout this paper we will refer to these coordinates as \emph{preferred local coordinates}.  From the local formulae, the preferred local coordinates show immediately that $\frac{\partial}{\partial x} g_{ij} = \frac{\partial}{\partial x}g^{T}_{ij} =0$.  Moreover, if $\phi$ evolves by the transverse parabolic Monge-Amp\`ere equation~(\ref{parabolic MA}), then local coordinate expressions for the evolved Sasaki metric $g(t)$ in the preferred local coordinates are obtained by replacing $h$ with $h+\frac{1}{2}\phi$.  We refer the reader to \cite{SmoWaZa} for more useful formulae which hold in these coordinates.

\section{The Transverse Energy Functional}

The aim of this section is to prove Theorem~\ref{trans F thm}.  In the course of the proof we will see that the transverse Ricci flow is related to the gradient flow of the $\mathcal{F}^{T}$ functional via diffeomorphisms.  This is in exact analogy with the Ricci flow.  However, if $(S,g)$ is Sasakian, the diffeomorphisms which relate the two flows need not preserve the transverse holomorphic structure of the manifold.  This necessitates a consideration of a larger class of manifolds, namely manifolds foliated by geodesics. Given a Riemannian foliation $(S,\xi)$, the class of metrics for which the foliation is generated by a Riemannian flow is described in the following definition.
\begin{definition}\label{met(S) def}
Let $(S,\xi)$  be a Riemannian foliation. We define the space $\mathfrak{Met}^{T}(S)$ to be
\begin{equation*}
\mathfrak{Met}^{T}(S,\xi) = \left\{ g\in C^{\infty}(S, S^{2}T^{*}M) | g>0, g(\xi,\xi)=1, \mathcal{L}_{\xi}g=0\right\}.
\end{equation*}
\end{definition}
\begin{remark}
Note that  if $g\in \mathfrak{Met}^{T}(S,\xi)$, then $g$ is bundle-like.  Moreover, it is a fact that if $g \in \mathfrak{Met}^{T}(S)$, then the integral curves of $\xi$ are geodesics.  In particular, for a general one-dimensional foliation the set $\mathfrak{Met}^{T}(S,\xi)$ may be empty.  However, in all applications we consider $\mathfrak{Met}^{T}(S,\xi)$ will not be empty.
\end{remark}
In order to lighten notation, when the vector field $\xi$ generating the foliation is understood, we will denote $\mathfrak{Met}^{T}(S,\xi)$ by $\mathfrak{Met}^{T}(S)$.  

By the above remarks, we see that every metric $g \in \mathfrak{Met}^{T}(S)$ induces a transverse metric $g^{T}$.  In fact, the opposite also holds.  Given a transverse metric $g^{T}$ on the quotient bundle $Q$, let $g \in \mathfrak{Met}^{T}(S)$ and define a metric $\tilde{g} \in \mathfrak{Met}^{T}(S)$ as follows; the metric $g$ defines an orthogonal decomposition $TS = L_{\xi} \oplus L_{\xi}^{\perp}$, hence any vector $X$ may be decomposed as $X = X^{\perp} + X^{\xi}$ where $X^{\perp} \in L_{\xi}^{\perp}$, and $X^{\xi} \in L_{\xi}$.  We then define the metric $\tilde{g}$ by
\begin{equation*}
\tilde{g}(X,Y) = g(X^{\xi}, Y^{\xi}) + g^{T}(X^{\perp}, Y^{\perp}).
\end{equation*}
It is an easy exercise to check that $\tilde{g} \in \mathfrak{Met}^{T}(S)$, and that $\tilde{g}$ induces the transverse metric $g^{T}$.

\begin{definition}\label{def F functional}
We define the transverse energy functional $\mathcal{F}^{T} : \mathfrak{Met}^{T}(S) \times C^{\infty}_{B}(S) \rightarrow \mathbb{R}$ by
\begin{equation*}
\mathcal{F}^{T}(g,f) = \int_{S} (\text{R}^{T} + |\nabla f|^{2})e^{-f}d\mu.
\end{equation*}
where $d\mu$ is the volume form associated to the metric $g$.
\end{definition}
\begin{remark}
We remark that while we have defined the functional $\mathcal{F}^{T}$ in terms of the full metric $g \in \mathfrak{Met}^{T}(S)$, it clearly only depends on the induced transverse metric.
\end{remark}
Our first task is to compute the variation of $\mathcal{F}^{T}$.  In order to do this, we need to define the space of variations which preserve $\mathfrak{Met}^{T}(S)$.  We observe that if $v\in C^{\infty}(S, S^{2}T^{*}S)$ has $v(\xi,\xi)=0$, and $\mathcal{L}_{\xi}v =0$, then for $t$ sufficiently small, $g+tv \in \mathfrak{Met}^{T}(S)$.  This motivates the following definition:
\begin{definition}
The (formal) tangent space of $\mathfrak{Met}^{T}(S)$, is the set
\begin{equation*}
T\mathfrak{Met}^{T}(S) = \left\{v\in C^{\infty}(S, S^{2}T^{*}S) : v(\xi,\xi)=0, \text{ and } \mathcal{L}_{\xi}v =0 \right\}.
\end{equation*}
\end{definition}
We can now compute the variation of the transverse energy functional.  

\begin{prop}\label{F var eqn}
If $g\in \mathfrak{Met}^{T}$, and $v\in T\mathfrak{Met}^{T}(S)$, and $h$ is a basic function then
\begin{equation*}
\delta \mathcal{F}^{T}(v,h) =\int_{S}e^{-f}\left(-v_{ij}\left(Ric^{T}_{ij}+\nabla_{i}\nabla_{j}f\right) + (v-h)(2\Delta_{B}f- |\nabla f|^{2} +R^{T})e^{-f} \right)d\mu.
\end{equation*}
\end{prop} 
\begin{proof}
Let $g \in \mathfrak{Met}^{T}(S)$, and fix $p \in S$.  By the remark following Definition~\ref{met(S) def}, we can find normal coordinates $\{x,y_{1}, \dots, y_{2n}\}$ in a neighbourhood $U$ of $p$ for the metric $g$ such that $\frac{\partial}{\partial x} = \xi$, and $\{\frac{\partial}{\partial y_{1}}|_{p}, \dots, \frac{\partial}{\partial y_{2n}}|_{p}\}$ is an orthonormal basis for $D(p)$.  In these coordinates, $g$ takes the form
\[ \left( \begin{array}{ccccc}
1 & g_{01} & g_{02} &\dots & g_{02n}  \\
g_{10} & \widetilde{g_{11}}+ g_{01}^{2} & \widetilde{g_{12}}+ g_{01}g_{02} & \dots & \widetilde{g_{12n}}+ g_{01}g_{02n}  \\
\vdots & \vdots & \vdots &\ddots&\vdots \\
g_{2n0}&  \widetilde{g_{2n1}}+ g_{02n}g_{01} &  \widetilde{g_{2n2}}+ g_{02n}g_{02} & \dots &  \widetilde{g_{2n2n}}+ g_{02n}^{2}  
\end{array} \right)\]
where $\widetilde{g_{ij}}$ is the $(i,j)$-th component of $g^{T}$, and we have used the index $0$ to denote the $dx$ components.  The variation $v$ is of the form
 \[ \left( \begin{array}{cccc}
0 & v_{01} & \dots & v_{02n}  \\
v_{10} & v_{11}& \dots & v_{12n} \\
\vdots & \vdots & \ddots&\vdots \\
v_{2n0}&  v_{2n1}& \dots & v_{2n2n}  
\end{array} \right)\]
where $\frac{\partial v_{ij}}{\partial x} =0$. In particular, observe that $\delta g^{T}_{ij}(p) = v_{ij}(p)$.  For $1 \leq i,j,k \leq 2n$.  The connection coefficients for the transverse Levi-Civita connection are
\begin{equation*}
\begin{aligned}
&\widetilde{\Gamma}_{ij}^{k} = \frac{1}{2} (g^{T})^{kl}\left(\partial_{i}g^{T}_{jl}+\partial_{j}g^{T}_{il} - \partial_{l}g^{T}_{ij}\right) \text{ for } 1 \leq i,j,k \leq 2n\\ &\widetilde{\Gamma}_{0j}^{k} = 0 \text{ for } 1 \leq j,k \leq 2n.
\end{aligned}
\end{equation*} 
Thus, at $p$, we have
\begin{equation*}
\begin{aligned}
\delta R^{T} = &\sum_{i,j =1}^{2n} -v_{ij} \left(\partial_{k}\widetilde{\Gamma}_{ij}^{k} - \partial_{j}\widetilde{\Gamma}_{ik}^{k}\right) + \partial_{k}\partial_{i}v_{ik} - \partial_{k}\partial_{k}v_{ii} \\&= \sum_{i,j =1}^{2n} -v_{ij} Ric^{T}_{ij} +\nabla_{k}\nabla_{i}v_{ik} - \Delta tr_{g}v_{i,j}
\end{aligned}
\end{equation*}
We now compute the variation of the volume form, to get
\begin{equation*}
\delta d\mu = g^{ij}v_{ij} d\mu = tr_{g}v_{ij} d\mu.
\end{equation*}
Setting $v= tr_{g}v_{ij} = tr_{g^{T}}v_{ij}$, it follows easily that
\begin{equation*}
\delta \mathcal{F}^{T}_{(v,h)} = \int_{S}e^{-f}\left(-v_{ij}\left(Ric^{T}_{ij}+\nabla_{i}\nabla_{j}f\right) + (v-h)(2\Delta_{B}f- |\nabla f|^{2} +R^{T}) \right)d\mu.
\end{equation*}
\end{proof}

Note that if $g^{T}$ and $f$ evolve by
\begin{equation}\label{pulled back metric}
\frac{\partial {g}^{T}}{\partial t} = -(Ric^{T} + D^{2}{f})
\end{equation}
\begin{equation}\label{pulled back dilaton}
\frac{\partial f}{\partial t} = -\Delta_{B}{f} - R^{T}
\end{equation}
\emph{and}, $(g(t),f(t)) \in \mathfrak{Met}^{T}(S) \times C^{\infty}_{B}(S)$, then $\mathcal{F}^{T}((g(t),f(t))$ is increasing, and
\begin{equation*}
\frac{d}{dt}\mathcal{F}^{T}(g(t),f(t))= \int_{S}|Ric^{T} +D^{2}f|^{2}e^{-f}d\mu.
\end{equation*}
As in the Ricci flow, the gradient flow equations~(\ref{pulled back metric}), and~(\ref{pulled back dilaton}), are related to the transverse Ricci flow via time-dependent diffeomorphisms.

\begin{prop}\label{coupled eqns prop}
Given a solution $(g(t),f(t)) \in \mathfrak{Met}^{T}(S) \times C^{\infty}_{B} $ to
\begin{equation*}
\frac{\partial g^{T}}{\partial t} = -Ric^{T},
\end{equation*}
\begin{equation}\label{TBHE}
\frac{\partial f}{\partial t} = -\Delta_{B}f + |\nabla f |^{2} - R^{T}
\end{equation}
with $(g(t),f(t))$ defined on $[0,T]$, $f(t)$ basic, $g(0)$ Sasakian, define a 1-parameter family of diffeomorphisms $\rho(t):S\rightarrow S$ by
\begin{equation}\label{ODE}
\frac{\partial \rho}{\partial t} = -\frac{1}{2}\nabla_{g(t)}f(t), \text{  } \text{  } \rho(0) = id_{S}
\end{equation}
which is a system of ODE admitting a solution on $[0,T]$.  Then the pulled back metrics $\bar{g}(t) = \rho(t)^{*}g(t)$ are bundle-like, and hence induce transverse metrics $\bar{g}^{T}(t)$.  The transverse metrics $\bar{g}^{T}(t)$ and the pulled back dilaton $\bar{f}(t) = \rho(t)^{*}f(t)$ satisfy equations~(\ref{pulled back metric}), and~(\ref{pulled back dilaton}) respectively.  Moreover, $\bar{f}$ is basic, and $\bar{g}(t) \in \mathfrak{Met}^{T}(S)$.
\end{prop}
\begin{remark}
The assumption that $g(t) \in \mathfrak{Met}^{T}(S)$ is not required.  It follows from the work in \cite{SmoWaZa}, that if $g(0)$ is Sasakian, and $g(t)$ solves~(\ref{T-Ricci}), then $g(t)$ is Sasakian, and hence $g(t) \in \mathfrak{Met}^{T}(S)$.  That equation~(\ref{ODE}) admits a solution of $[0,T]$ is standard; see, for example, Lemma 3.15 in \cite{BenChow1}.
\end{remark}

The proof of Proposition~\ref{coupled eqns prop} will follow essentially from the following lemma.

\begin{Lemma}\label{time dep. diffeo Lemma}
Under the assumptions of Proposition~\ref{coupled eqns prop}, for each $t\in [0,T]$, we have 
\end{Lemma}
\begin{enumerate}
\item[{(\it i)}]
$\rho(t)_{*}\xi=\xi$. 
\item[{(\it ii)}]
$\rho(t)^{*}g$ is bundle-like for the foliation induced by the vector field $\xi$.  Moreover, $\rho(t)^{*}(g^{T}) = (\rho(t)^{*}g)^{T}$.
\end{enumerate}

\begin{proof}
For the rest of the proof, fix a point $p\in S$ and choose preferred local coordinates $\{x,z_{1},\dots,z_{n}, \bar{z}_{1},\dots, \bar{z}_{n}\}$ in an open neighbourhood $U \ni p$, for the metric $g(0)$.

We begin by proving {\it(i)} for small time; fix a time $t$ and asume that $t$ is chosen sufficiently small so that $\rho(s)(x,0,\dots,0)$ remains in the coordinate patch for every $x \in [-\epsilon, \epsilon]$ for some $\epsilon >0$, and $0\leq s \leq 2t$. Observe that $f(t)$ and $g(t)$ are independent of $x$, and so the time dependent vector field $ -\frac{1}{2}\nabla_{g(t)}f(t)$ is independent of $x$.  The curves $\rho(s)(x,0,\dots,0)$ have the same tangent vectors for every $x \in [-\epsilon, \epsilon]$ and $s \in [0,t]$. Hence $\rho(t)(x,0,\dots,0) = (x,0,\dots,0)+\rho(t)(0,0,\dots,0)$.  It is clear then that $\rho(t)_{*}\xi = \xi$, proving the result for small $t$.  The result for general $t$ is obtained by taking a sequence of preferred local coordinates in open sets $U_{i}$ which cover the curve $\rho(s)$ where $s\in [0,t]$, and applying the above argument repeatedly.

If $h$ is a smooth, basic, local function in a neighbourhood of $\rho(t)(p)$, then $\rho(t)^{*}h$ is a smooth, basic function in a neighbourhood of $p$, since $\xi \rho^{*} h = \rho_{*}\xi h = \xi h = 0$.  Using the local formula for the evolved contact form, and the above preferred coordinates, this yields
\begin{equation*}
 \rho(t)^{*}\eta(t) = dx + \sqrt{-1}\sum_{j=1}^{n} f_{j}dz^{j} -\sqrt{-1}\sum_{j=1}^{n} f_{\bar j} d\bar{z}^{j},
 \end{equation*}
where $f_{j}, f_{\bar j}$ are smooth, basic, local functions.  A similar argument provides a local formula for $\rho(t)^{*}g(t)$.  They key point is that the component functions of $ \rho(t)^{*}\eta(t)$, and $\rho(t)^{*}g(t)$ are basic in the preferred coordinates.  From now on, we suppress the argument $t$.  We have the following formula for the metric
\begin{equation}\label{failure to stay Sasaki}
\rho^{*}g = \rho^{*}(\eta\otimes\eta + d\eta(\cdot, \Phi \cdot)) = \rho^{*}\eta \otimes \rho^{*}\eta + d\eta(\rho_{*}\cdot, \Phi \rho_{*}\cdot).
\end{equation}
In particular, $\rho^{*}g(\xi, \cdot) = \rho^{*}\eta$.  We can now prove {\it(ii)}.  Suppose that $X$ is a foliate vector field orthogonal $\xi$, then
\begin{equation*}
X= A(x,z_{1},\dots,z_{n},\bar{z}_{1},\dots, \bar{z}_{n})\frac{\partial}{\partial x} +Z(z_{1},\dots,z_{n},\bar{z}_{1},\dots, \bar{z}_{n}),
\end{equation*}
where
\begin{equation*}
Z = \sum_{i=1}^{n}B_{i}(z_{1},...,z_{n},\bar{z}_{1},..., \bar{z}_{n})\frac{\partial}{\partial z_{i}} +\sum_{i=1}^{n}C_{i}(z_{1},...,z_{n},\bar{z}_{1},..., \bar{z}_{n})\frac{\partial}{\partial \bar{z}_{i}}.
\end{equation*}
Since $X$ is orthogonal $\xi$, and $\rho_{*}\xi = \xi$ we have $0=\rho^{*}g(X,\xi) = \rho^{*}\eta(X) = A +\rho^{*}\eta(Z)$, and so
\begin{equation}\label{bundle like relation}
A(x,z_{1},\dots,z_{n},\bar{z}_{1},\dots, \bar{z}_{n}) = -\rho^{*}\eta(Z).
\end{equation}
By the above arguments, the right hand side of ~(\ref{bundle like relation}) is a basic function, and thus $A$ is independent of $x$.  It follows that $\rho^{*}g(X,Y)$ is a basic function for any vector fields $X, Y$ verifying the assumptions of Definition~\ref{bundle-like}.  The last statement is an easy application of the definition of $g^{T}$, and property {\it (i)}.
\end{proof}

Note that in equation~(\ref{failure to stay Sasaki}), it is precisely the fact that $\Phi \rho_{*} \ne \rho_{*} \Phi$ which prevents the pulled back metric $\rho^{*}g$ from being Sasakian, and necessitates our consideration of the large class of manifolds foliated by Riemannian flows.  We can now prove Proposition~\ref{coupled eqns prop}.

\begin{proof}[Proof of Proposition~\ref{coupled eqns prop}.]
That $\bar{g}$ is bundle-like follows from Lemma~\ref{time dep. diffeo Lemma}.  That $\bar{f}$ is basic is clear.  That equations~(\ref{pulled back metric}) and~(\ref{pulled back dilaton}) are satisfied is a standard computation.  That $\bar{g} \in \mathfrak{Met}^{T}(S)$ follows from Lemma~\ref{time dep. diffeo Lemma}, since $\mathcal{L}_{\xi}\rho^{*}g = \mathcal{L}_{\xi}g$ if $\rho_{*}\xi = \xi$.
\end{proof}

In order to prove Theorem~\ref{trans F thm}, we are thus reduced to showing that we can solve the coupled equations~(\ref{T-Ricci}) and~(\ref{TBHE}).

\begin{prop}\label{solve TBHE}
If $g^{T}(t)$ is a transverse Ricci flow on $[0,T]$ with $g(0)$ Sasakian, and $f_{T}$ is a basic function, then $\exists f(t)$ for $t \in [0,T]$ solving the transverse backward heat equation~(\ref{TBHE}) with $f(t)$ basic for each time $t$, and $f(T)=f_{T}$.
\end{prop}
\begin{proof}
Let $\{U_{\alpha}\}$ be a cover of $S$ by preferred local coordinate charts, constructed with respect to the metric $g(0)$, and let $\{V_{\alpha}\}$ be also as before.  Set $u(s) = e^{-f(s)}$, where $s=T-t$.  Then the equation~(\ref{TBHE}) becomes
\begin{equation*}
\frac{\partial u}{\partial s} = \Delta_{B} u -R^{T} u, \text{  } \text{  } u(0) = \hat{u}.
\end{equation*}
Fix an open set $U_{\beta}$.  Note that by the existence and uniqueness results of \cite{SmoWaZa}, $(g^{T})^{-1}$ and $R^{T}$ are independent of $x$.  By assumption, the initial condition $ \hat{u}$ is independent of $x$ also.  Thus, if we find a solution $\tilde{u}$ to
\begin{equation}\label{Restricted Heat}
\frac{\partial \tilde{u}}{\partial s} = \Delta \tilde{u} -R^{T} \tilde{u}, \text{  } \text{  } \tilde{u}(0) = \hat{u},
\end{equation}
on $V_{\beta} \subset \mathbb{C}^{n}$ and then set $u(x,z_{1},\dots z_{n}) = \tilde{u}(z_{1},\dots,z_{n})$, then this will solve the problem on $U_{\beta}$.  Now, equation~(\ref{Restricted Heat}) is parabolic as long as $g^{T}$ remains positive definite, and hence we can solve~(\ref{Restricted Heat}) on $[0,T]$.  We have thus generated a family of local basic solutions $\{ U_{\alpha}, u_{\alpha} \}$.  It remains to show that these solutions glue to a global solution.  This is a consequence of the following lemma.
 
\begin{Lemma}
Let $w(t)$ be a basic function solving~(\ref{Restricted Heat}) on $U_{\alpha} \cap U_{\beta}$ with $w(0)=0$.  Then $w(t) = 0$ $ \forall t\in [0,T]$.
\end{Lemma}
\begin{proof}
Since $R^{T}$ is a smooth function on $S \times [0,T]$, we may choose $\lambda >0$ so that $R^{T} + \lambda >0$ on $S \times [0,T]$.  Set $v = e^{-\lambda s}w$, and observe that $v$ satisfies
\begin{equation*}
\frac{\partial v}{\partial s} = \Delta v -(R^{T} + \lambda)v
\end{equation*}
where we are considering the equation on $V \subset \mathbb{C}^{n}$.  Applying the maximum principle to $v$ and $-v$, and using that $v$ is basic proves the lemma.
\end{proof}

Now, the lemma shows that the solutions glue to a global solution, and it is clear, by construction, that the solution is basic.
\end{proof}

We now investigate the dependence of the functional $\mathcal{F}^{T}$ under diffeomorphisms of the type constructed in Proposition~\ref{coupled eqns prop}, whose existence we have just demonstrated in Proposition~\ref{solve TBHE}.  We begin by abstracting the properties of these diffeomorphisms.  

\begin{definition}
Let $\mathfrak{Diff}$ be the group of diffeomorphisms of $S$.  We define the class of transverse diffeomorphisms $\mathfrak{Diff}^{T}$
\begin{equation*}
\mathfrak{Diff}^{T} = \left\{ \rho \in \mathfrak{Diff} : \rho_{*}\xi=\xi \text{ and } \rho^{*}(g^{T}) = (\rho^{*}g)^{T}\right\}
\end{equation*}
\end{definition}

First note that $\mathfrak{Diff}^{T}$ is a subgroup of $\mathfrak{Diff}$ and that $\mathfrak{Diff}^{T}$ preserves the class $\mathfrak{Met}^{T}$.  We now show that the group $\mathfrak{Diff}^{T}$ preserves the transverse scalar curvature.

\begin{Lemma}\label{scalar curv invar}
Let $g \in \mathfrak{Met}^{T}(S)$.  If $\rho \in \mathfrak{Diff}^{T}$ then $\rho^{*}R^{T}(g) = R^{T}(\rho^{*}g)$
\end{Lemma}
\begin{proof}
Fix a point $s \in S$.  Let $\{ E_{i} \}$ be a local orthonormal frame for $L_{\xi}^{\perp}$.  By \cite{BoyGal1}, we have the following formula for any metric $g \in \mathfrak{Met}^{T}(S)$,
\begin{equation*}
R^{T} = R + 2\sum_{i} g^{T}(\nabla_{E_{i}} \xi, \nabla_{E_{i}}).
\end{equation*}
Where $R$ is the scalar curvature of the metric $g$.  Observe that as $\mathfrak{Diff}^{T}$ preserves $\mathfrak{Met}^{T}(S)$, we have
\begin{equation}\label{diff in scalar curv}
\rho^{*}R^{T}(g) - R^{T}(\rho^{*}g) =2\sum_{i} \rho^{*}(g^{T}(\nabla_{E_{i}} \xi, \nabla_{E_{i}}\xi)) - (\rho^{*}g)^{T}(\tilde{\nabla}_{\rho_{*}^{-1}E_{i}} \xi, \tilde{\nabla}_{\rho_{*}^{-1}E_{i}}\xi).
\end{equation}
where we have used $\tilde{\nabla}$ to denote the Levi-Civita connection of $\rho^{*}g$, and that $\{ \rho_{*}^{-1}E_{i} \}$ is an orthonormal frame with respect to $\rho^{*}g$.  It is important to note that the first term on the right hand side of~(\ref{diff in scalar curv}) is the pullback of a \emph{function}.  In particular, we have
\begin{equation*}
\rho^{*}(g^{T}(\nabla_{E_{i}} \xi, \nabla_{E_{i}}\xi)) = \rho^{*}g^{T}(\rho_{*}^{-1}\nabla_{E_{i}} \xi, \rho_{*}^{-1}\nabla_{E_{i}} \xi).
\end{equation*}
It is elementary to check that
\begin{equation*}
\tilde{\nabla}_{X}Y = \rho_{*}^{-1}\left( \nabla_{\rho_{*}X}\rho_{*}Y \right).
\end{equation*}
Thus, since $\rho \in \mathfrak{Diff}^{T}$ we see that  the right hand side of~(\ref{diff in scalar curv}) is zero.
\end{proof}

\begin{prop}\label{diff invar prop}
The functional $\mathcal{F}^{T}$ is invariant under the action of $\mathfrak{Diff}^{T}$.
\end{prop}
\begin{proof}
By Lemma~\ref{scalar curv invar}, it suffices to show that if $f \in C^{\infty}_{B}(S)$ then so is $\rho^{*}f$.  This last fact is obvious as $\rho_{*}\xi = \xi$.  Thus we have
\begin{equation*}
|\nabla f|_{g^{T}} = |\nabla f |_{g} = |\nabla \rho^{*}f|_{\rho^{*}g} = |\nabla \rho^{*}f|_{(\rho^{*}g)^{T}} = |\nabla \rho^{*}f|_{\rho^{*}(g^{T})}
\end{equation*}
\end{proof}
We can now prove Theorem~\ref{trans F thm}.

\begin{proof}[Proof of Theorem~\ref{trans F thm}]
By Proposition~\ref{solve TBHE} we can solve equation~(\ref{TBHE}).  By Proposition~\ref{coupled eqns prop} the pulled back pair $(\bar{g}(t), \bar{f}(t))$ satisfy the gradient flow equations~(\ref{pulled back metric}) and~(\ref{pulled back dilaton}), and hence $\mathcal{F}^{T}(\bar{g}(t), \bar{f}(t))$ is increasing.   By Lemma~\ref{time dep. diffeo Lemma}, the time dependent diffeomorphisms $\rho(t) \in \mathfrak{Diff}^{T}$ and so by Proposition~\ref{diff invar prop}, $\mathcal{F}^{T}(g,f) = \mathcal{F}^{T}(\bar{g}, \bar{f})$.  The theorem is proved.
\end{proof}


\section{The Transverse Entropy Functional}
The aim of this section is to prove Theorem~\ref{trans W thm}.  As with the Ricci flow, the proof of the theorem is essentially the same as the proof of Theorem~\ref{trans F thm}, and hence we will omit the details.  We refer the reader to \cite{BenChow, Pman} for the details in the Ricci flow case.
\begin{definition}\label{W def}
Let $(S,\xi)$ be a foliated Riemannian manifold.  Define the transverse entropy functional $\mathcal{W}^{T}: \mathfrak{Met}^{T}(S) \times C^{\infty}_{B}(S) \times \mathbb{R}_{>0} \rightarrow \mathbb{R}$ by
\begin{equation*}
\mathcal{W}^{T}(g,f,\tau) =(4\pi \tau)^{-n}\int_{S}\left( \tau(\text{R}^{T} + |\nabla f|^{2}) + (f-2n) \right) e^{-f}d\mu
\end{equation*}
where $d\mu$ is the volume form associated to the metric $g$.
\end{definition}
\begin{remark}
Before proceeding, we make a few observations about the scale invariance of $\mathcal{W}^{T}$.  By the remarks following Definition~\ref{def F functional}, we see that $\mathcal{W}^{T}$ depends only on the induced transverse metric.  Moreover, it is easy to see that if $g, \tilde{g} \in \mathfrak{Met}^{T}(S)$ induce the transverse metrics $g^{T}, cg^{T}$ respectively, then $\mathcal{W}^{T}(g,f,\tau) = \mathcal{W}^{T}(\tilde{g},f,c\tau)$.  By abuse of notation, we will sometimes consider the functional $\mathcal{W}^{T}(g^{T},f,\tau)$, by which we mean $\mathcal{W}^{T}(g,f,\tau)$ for any metric $g \in \mathfrak{Met}^{T}(S)$ which induces $g^{T}$.  Such a metric $g$ exists by the discussion following Definition~\ref{met(S) def}.
\end{remark}
One can compute the variational equations for $\mathcal{W}^{T}$ easily by using the computation for the $\mathcal{F}^{T}$ functional and following the same steps as in \cite{BenChow}.  In particular, one sees that if $(g(t), f(t), \tau(t))$ evolve by
\begin{equation}\label{pulled back entropy metric}
\frac{\partial {g}^{T}}{\partial t} = -(Ric^{T} + D^{2}{f})
\end{equation}
\begin{equation}\label{pulled back entropy dilaton}
\frac{\partial f}{\partial t} = -\Delta_{B}{f} - R^{T} +\frac{n}{\tau}
\end{equation}
\begin{equation}\label{tau evolution}
\frac{d \tau}{dt}=-1
\end{equation}
then $\mathcal{W}^{T}(g(t),f(t),\tau(t))$ is increasing, and $\int_{S}(4\pi \tau)^{-n}e^{-f}d\mu=1$.  By an elementary modification of the proof of Proposition~\ref{coupled eqns prop} we have

\begin{prop}\label{Entropy coupled eqns prop}
Suppose $(g(t),f(t), \tau(t)) \in \mathfrak{Met}^{T}(S) \times C^{\infty}_{B}\times\mathbb{R}_{>0} $ solves
\begin{equation*}
\frac{\partial g^{T}}{\partial t} = -Ric^{T}
\end{equation*}
\begin{equation}\label{Entropy TBHE}
\frac{\partial f}{\partial t} = -\Delta_{B}f + |\nabla f |^{2} - R^{T}+\frac{n}{\tau}
\end{equation}
\begin{equation*}
\frac{d \tau}{dt}=-1
\end{equation*}
with $(g(t),f(t))$ defined on $[0,T]$, $f(t)$ basic, $g(0)$ Sasakian, and $\tau(t) >0$.  Define a 1-parameter family of diffeomorphisms $\rho(t):S\rightarrow S$ by
\begin{equation*}
\frac{\partial \rho}{\partial t} = -\frac{1}{2}\nabla_{g(t)}f(t), \text{  } \text{  } \rho(0) = id_{S}
\end{equation*}
which is a system of ODE admitting a solution on $[0,T]$.  Then the pulled back metrics $\bar{g}(t) = \rho(t)^{*}g(t)$ are bundle-like, and hence induce transverse metrics $\bar{g}^{T}(t)$.  The transverse metrics $\bar{g}^{T}(t)$ and the pulled back dilaton $\bar{f}(t) = \rho(t)^{*}f(t)$ satisfy equations~(\ref{pulled back entropy metric}), and~(\ref{pulled back entropy dilaton}) respectively.  Moreover, $\bar{f}$ is basic, and $\bar{g}(t) \in \mathfrak{Met}^{T}(S)$.
\end{prop}
Easy modifications of the proofs of Propositions~\ref{coupled eqns prop} and~\ref{diff invar prop} yield
\begin{prop}\label{solve Entropy TBHE}
If $g^{T}(t)$ is a transverse Ricci flow on $[0,T]$ with $g(0)$ Sasakian, $\tau(t)$ a positive solution of~(\ref{tau evolution}),  and $f_{T}$ is a basic function, then $\exists f(t)$ for $t \in [0,T]$ solving the transverse backward heat equation~(\ref{Entropy TBHE}) with f(t) basic for each time $t$, and $f(T)=f_{T}$.
\end{prop}

\begin{prop}\label{W diff invar prop}
The functional $\mathcal{W}^{T}$ is invariant under the action of $\mathfrak{Diff}^{T}$.
\end{prop}

We can now prove Theorem~\ref{trans W thm} by mimicking the proof of Theorem~\ref{trans F thm}.
\begin{proof}[Proof of Theorem~\ref{trans W thm}]
By Proposition~\ref{solve Entropy TBHE} we can solve equation~(\ref{Entropy TBHE}).  By Proposition~\ref{Entropy coupled eqns prop} the pulled back triple $(\bar{g}(t), \bar{f}(t), \tau(t))$ satisfy the gradient flow equations~(\ref{pulled back entropy metric}),~(\ref{pulled back entropy dilaton}) and~(\ref{tau evolution}) and hence $\mathcal{W}^{T}(\bar{g}(t), \bar{f}(t), \tau(t))$ is increasing.   By Lemma~\ref{time dep. diffeo Lemma}, the time dependent diffeomorphisms $\rho(t) \in \mathfrak{Diff}^{T}$ and so by Proposition~\ref{W diff invar prop}, $\mathcal{W}^{T}(g,f, \tau) = \mathcal{W}^{T}(\bar{g}, \bar{f}, \tau)$.  The theorem is proved.
\end{proof}

We now define the transverse analogue of Perelman's $\mu$-functional. 

\begin{definition}
Set 
\begin{equation*}
\chi = \left\{(g, f, \tau) \in \mathfrak{Met}^{T}(S) \times C^{\infty}_{B}(S) \times \mathbb{R}_{>0} : \int_{S}(4\pi \tau)^{-n}e^{-f}d\mu=1\right\}.
\end{equation*}
Then define the functional $\mu^{T} :\mathfrak{Met}^{T}(S) \times \mathbb{R}_{>0}$ by
\begin{equation*}
\mu^{T}(g,\tau) = \inf \left\{ \mathcal{W}^{T}(g, f, \tau) : (g, f, \tau) \in \chi\right\}.
\end{equation*}
\end{definition}
\begin{remark}
It is easy to check from the definition that $\mu^{T}$ is invariant under the action of $\mathfrak{Diff}^{T}$.  Moreover, from the definition,  $\mu^{T}$ shares the same scale invariance as $\mathcal{W}^{T}$; we refer the reader to the remark following Definition~\ref{W def}.
\end{remark}

In analogy with the Ricci flow, the $\mu^{T}$-functional is always finite.
\begin{prop}
For any pair $(g, \tau) \in \mathfrak{Met}^{T}(S) \times \mathbb{R}_{>0}$,
\begin{equation*}
\mu^{T}(g, \tau) > -\infty.
\end{equation*}
\end{prop}
\begin{proof}
The proof is identical to the proof for the general Ricci flow.  We refer the reader to \cite{BenChow, Roth}, and the appendix.
\end{proof}

\begin{prop}
Let $(S,\xi)$ be a foliated Riemannian manifold, $g(t)$ a solution of the transverse Ricci flow on $[0,T]$ with $g(0)$ Sasakian.  Suppose that $\tau(t)$ is a solution of~(\ref{tau evolution}).  Then $\mu^{T}(g(t), \tau(t))$ is increasing.
\end{prop}
\begin{proof}
Let $t_{0} \in [0,T]$, and $f(t_{0}) \in C^{\infty}_{B}(S)$ such that $(g(t_{0}), f(t_{0}), \tau(t_{0})) \in \chi$.  Let $f(t)$ be a solution of~(\ref{Entropy TBHE}).  Then, by the monotonicity of $\mathcal{W}^{T}$ we have
\begin{equation*}
\mu^{T}(g(0),\tau(0)) \leq \mathcal{W}^{T}\left(g(0), f(0), \tau(0)\right) \leq \mathcal{W}^{T}\left(g(t_{0}), f(t_{0}), \tau(t_{0})\right).
\end{equation*}
Taking the infimum of the last quantity yields the result.
\end{proof}

Throughout this section, and the last, we assumed that the evolving metric $g(t)$ was a solution of the transverse Ricci flow.  In the remaining sections, we will be working with the normalized version of the transverse Ricci flow, which we have been referring to as the Sasaki-Ricci flow.  Observe that if $g^{T}(s)$ is a solution of the Sasaki-Ricci flow, then $\tilde{g}^{T}(t) = (1-t)g^{T}(-\log(1-t))$ is a solution of the transverse Ricci flow.  Thus, if we set $\tau(t) = 1-t$ and let $f(t)$ evolve according to~(\ref{Entropy TBHE}) then the transverse scale invariance of $\mathcal{W}^{T}$ implies that
\begin{equation*}
\mathcal{W}^{T}((1-t)g^{T}(-\log(1-t)), f(t), 1-t) = \mathcal{W}^{T}(g^{T}(-\log(1-t),f(t), 1)
\end{equation*}
is increasing.  In particular, we have that $\mu^{T}(g^{T}(s),1)$ is increasing along the Sasaki-Ricci flow.  An important property of the $\mu^{T}$ functional is:

\begin{Lemma}\label{minimizer}
Let $(S,g)$ be a Sasakian manifold, and let $\tau >0$.  There exists $f_{\tau} \in C^{\infty}_{B}(S)$ so that $\mathcal{W}^{T}(g, f_{\tau}, \tau) = \mu^{T}(g, \tau)$.
\end{Lemma}
The proof is given in the appendix. 
\begin{remark}
A consequence of the proof of Lemma~\ref{minimizer} is that
\begin{equation*}
\mu^{T}(g,\tau) = \inf \left\{ \mathcal{W}^{T}(g,f,\tau) : f \in W^{1,2}_{B}(S), \int_{S}(4\pi \tau)^{-n}e^{-f}d\mu=1\right\}.
\end{equation*}
We refer the reader to the appendix for the definition of $W^{1,2}_{B}(S)$.
\end{remark}

\section{Bounds along the Sasaki-Ricci flow}
In this and the following sections we aim to prove Theorem~\ref{Pman thm}.  The proof has essentially two parts.  In the first part, we employ maximum principle techniques to show that the conclusion of Theorem~\ref{Pman thm} follows from a uniform transverse diameter bound along the Sasaki-Ricci flow.  In the second part of the proof we use the functionals $\mu^{T}$, and $\mathcal{W}^{T}$ to prove a non-collapsing theorem, which is of independent interest.  We then employ the non-collapsing theorem to obtain the required diameter bound.  The arguments in this section and the next are direct adaptations of Perelman's arguments, which can be found in \cite{Pman, SesumTian}.

Let $\phi(t)$ be a solution to the Sasaki-Ricci flow~(\ref{SRF}) with initial condition $\phi(0)=0$. Using the transverse $\partial \bar{\partial}$-lemma of \cite{ElKac}, there is a basic function $u$ such that
\begin{equation*}
 \partial_{l}\partial_{\bar{k}}\dot{\phi}=\frac{\partial g^{T}_{\bar{k}l}}{\partial t} = g^{T}_{\bar{k}l} - R^{T}_{\bar{k}l} = \partial_{l}\partial_{\bar{k}}u
\end{equation*}
We may assume that $\dot{\phi}=u$ and that $u$ is normalized by the condition
\begin{equation}\label{normalization}
\int_{S} e^{-u(t)} d\mu = (4\pi)^{n}.
\end{equation}
We compute that 
\begin{equation*}
\partial_{l}\partial_{\bar{k}}(\frac{\partial u}{\partial t}) =  \partial_{l}\partial_{\bar{k}}u +  \partial_{l}\partial_{\bar{k}}\Delta_{B}u.
\end{equation*}
Thus, we may take $u$ to evolve by $\dot{u} = \Delta_{B}u +u-a$, where $a=(4\pi)^{-n}\int ue^{-u}$.


\begin{Lemma}
The quantity $a$ is monotone under the Sasaki-Ricci flow.  In particular, there is a uniform constant $C_{1}$ depending only on $g(0)$ such that $a=(4\pi)^{-n} \int_{S} ue^{-u} d\mu \geq C_{1}$
\end{Lemma}
\begin{proof}
Using the evolution equation for $u$, we compute
\begin{equation*}
\dot{a} = (4\pi)^{-n}\int_{S}\left( \Delta_{B}u +u-a-u^{2}-u\Delta_{B}u+ua +u\Delta_{B}u\right)e^{-u}d\mu.
\end{equation*}
Using now the definition of $a$ and the normalization~(\ref{normalization}), we get upon integration by parts
\begin{equation*}
\dot{a} = a^{2}+(4\pi)^{-n}\int_{S} |\nabla u|^{2}e^{-u}d\mu -(4\pi)^{-n}\int_{S}u^{2}e^{-u}d\mu.
\end{equation*}
Thus, it suffices to prove the following Poincar\'{e} type inequality:
\end{proof}

\begin{Lemma}
Let $u$ satisfy the equation $g^{T}_{\bar{k}j}-Ric^{T}_{\bar{k}j} = \partial_{j}\partial_{\bar{k}}u$.  The following inequality holds for all $f\in C^{\infty}_{B}(S)$;
\begin{equation}
\frac{1}{Vol(S)}\int_{S}f^{2}e^{-u}d\mu \leq \frac{1}{Vol(S)}\int_{S}|\nabla f|^{2}e^{-u}d\mu +\left(\frac{1}{Vol(S)}\int_{S}f e^{-u} d\mu\right)^{2}.
\end{equation}

\end{Lemma}
The proof is given in the appendix.  The quantity $a$ is trivially bounded above, as $xe^{-x}$ is bounded above on $\mathbb{R}$. 

\begin{Lemma}
There is a uniform constant $C_{2}>0$ such that $a\leq C_{2}$.
\end{Lemma}
 The following lemma shows that it is enough to obtain a bound from above for the scalar curvature.
\begin{Lemma}
The transverse scalar curvature $R^{T}$ is uniformly bounded from below along the Sasaki-Ricci flow.
\end{Lemma}
\begin{proof}
We compute the evolution equation for the transverse scalar curvature
\begin{equation*}
\dot{R^{T}} = -R^{T} + |Ric^{T}|^{2} + \Delta_{B}R^{T}.
\end{equation*}
As $R^{T}$ is basic, an application of the minimum principle yields the result.
\end{proof}
We now use the uniform bound for $a$ to obtain a uniform lower bound for $u$, which reduces the problem to obtaining uniform upper bounds for the transverse scalar curvature, and $u$.
\begin{Lemma}\label{uniform bdd below}
The function $u(t)$ is uniformly bounded below.
\end{Lemma}
\begin{proof}
The proof is identical to the proof in \cite{SesumTian}, and we may use that $\Delta_{B} = \Delta$ on basic functions.
\end{proof}

\begin{prop}\label{bounds by Cu}
There is a uniform constant $C$ such that $|\nabla u|^{2} + |R^{T}| \leq C(u+C)$.
\end{prop}
\begin{proof} The computations and arguments in the proof of this proposition in \cite{SesumTian} are completely local.  Thus they carry over verbatim to the Sasaki setting.
\end{proof}

\begin{Lemma}\label{bounds by Cd}
Let $x\in S$ be such that $u(x,t) = min_{y\in S}\text{  } u(y,t)$.  There is a uniform constant $C$ such that
\begin{equation*}
u(x,t) \leq C d^{T}(x,y)^{2} + C
\end{equation*}
\begin{equation*}
R^{T}(x,t) \leq Cd^{T}(x,y)^{2} +C
\end{equation*}
\begin{equation*}
|\nabla u| \leq Cd^{T}(x,y)^{2} +C
\end{equation*}
\end{Lemma}
\begin{proof}
By Lemma~\ref{uniform bdd below} we can assume $u \geq \delta >0$.  From Proposition~\ref{bounds by Cu}, we have that $\sqrt{u}$ is uniformly Lipschitz bounded and basic.  Thus,
\begin{equation*}
|\sqrt{u(y,t)}-\sqrt{u(z,t)}| \leq \frac{|\nabla u|(p,t)}{2\sqrt{u}}d^{T}_{t}(y,z)\leq Cd^{T}_{t}(y,z)
\end{equation*}
Thus,
\begin{equation*}
u(y,t)\leq C_{1}d^{T}_{t}(x,y)^{2} + C_{1}u(x,t)^{2}.
\end{equation*}
Now, $u(x,t) \leq K$ for some $K$ independent of $t$, for if not, then
\begin{equation*}
(4\pi)^{n} = \int_{S}e^{-u}dV_{t} \leq e^{-u(x,t)}Vol(S)\rightarrow 0.
\end{equation*}
In particular, $u(y,t)\leq Cd^{T}_{t}(y,x)^{2}+C'$ for $C,C'$ independent of $t$.  The conclusion of the lemma follows from this and Proposition~\ref{bounds by Cu}.
\end{proof}

\section{Non-Collapsing and Uniform Diameter Bounds}
In this section we use the transverse entropy functional to prove a non-collapsing theorem for the Sasaki-Ricci flow, which we then use to obtain uniform diameter bounds along the Sasaki-Ricci flow.  The proof of the non-collapsing theorem is essentially that of Perelman \cite{Pman} (see also \cite{BenChow, KleinerLott, SesumTian}).  
\begin{prop}\label{non-collapse 1}
Let $g^{T}(t)$ be a solution of the Sasaki-Ricci flow.  There exists a positive constant $C$, depending only on $g(0)$, such that for every $p \in S$, $Vol(T_{g(t)}(x,1))\geq C$, for time $t$ where the metric $g(t)$ satisfies $|R^{T}| \leq1$ on $T_{g(t)}(x,1)$.
\end{prop}
Note that if the condition $|R^{T}| \leq1$ holds at a point $x \in S$, then it holds everywhere on $orb_{\xi}x$ as the transverse scalar curvature is constant along $\xi$.  Thus, the conditions in Proposition~\ref{non-collapse 1} are local.  That is, after fixing preferred local coordinates, it suffices to check the condition on a single fibre $\{x=const\}$.  The proof follows from the following useful Proposition.

\begin{prop}\label{non-collapse 2}
Let $g^{T}(t)$ be a solution of the unnormalized Sasaki-Ricci flow $(d/dt) g(t) = -Ric^{T}(g(t))$.  There is a constant $\kappa = \kappa(g(0)) >0$ so that if,  $ |R^{T}(g(t))| \leq 1/r^{2} $ in a tube $T_{g(t)}(p,r)$ around a point $p$ with $\overline{orb_{\xi}p}$ a torus of dimension $q$ then $Vol_{g_{(t)}}(T_{g(t)}(p,r))>\kappa r^{2n}.$
\end{prop}
\begin{proof}
We argue by contradiction.  Suppose there exists sequence of points $p_{k}\in S$, times $t_{k}\rightarrow T$, and radii $r_{k}$ so that $\dim \overline{orb_{\xi}p_{k}}=q$, and
\begin{equation}\label{collapse cond}
|R_{k}^{T}|=|R_{g(t_{k})}^{T}| \leq \frac{C}{r_{k}^{2}}, \text{ but }  \text{  } r_{k}^{-2n}Vol_{g_{(t_{k})}}(T_{g(s_{k})}(p_{k},r_{k})) = r_{k}^{-2n}Vol(T_{k}) \rightarrow 0
\end{equation}

\begin{Lemma}\label{radius Lemma}
Fix $p \in S$ and $t \in [0,T)$, and suppose $\exists r > 0 $ such that $|R^{T}|=|R_{g(t)}^{T}| \leq \frac{C}{r^{2}}$ on $T(p,r)$.  Then there exists $r' \in (0,r]$ such that
\begin{enumerate}
\item[{\it (i)}] $|R^{T}|=|R_{g(t)}^{T}| \leq \frac{C}{r'^{2}}$ on $T(p,r')$
\item[{\it (ii)}] $(r')^{-2n}Vol(T(p,r'))\leq 3^{2n}r^{-2n}Vol(T(p,r)) $
\item[{\it (iii)}] $Vol_{g(t)}(T_{g(t)}(p,r')) - Vol_{g(t)}(T_{g(t)}(p,r'/2))\leq C(n,q)$.
\end{enumerate}
\end{Lemma}
\begin{proof}
By Proposition~\ref{transverse v geo prop} and the volume expansion in Theorem~\ref{Gray Thm}, we know that 
\begin{equation*}
\lim_{k\rightarrow \infty} \frac{Vol_{g(t)}(T_{g(t)}(p,r/2^{k})) }{Vol_{g_{(t)}}(T_{g(t)}(p,r/2^{k+1}))} = 2^{2n+1-q}
\end{equation*}
Hence, there is a $k < \infty$ such that 
\begin{enumerate}
\item[\textbullet] 
\begin{equation*}
\frac{Vol_{g(t)}(T_{g(t)}(p,r/2^{k}))}{Vol_{g(t)}(T_{g(t)}(p,r/2^{k+1}))}\leq 3^{2n}
\end{equation*}
\item[\textbullet] if $l < k$, then
\begin{equation}\label{iterate}
 \frac{Vol_{g(t)}(T_{g(t)}(p,r/2^{l})) }{Vol_{g(t)}(T_{g(t)}(p,r/2^{l+1}))}>3^{2n}
 \end{equation}
 \end{enumerate}
By iterating the inequality in~(\ref{iterate}) we obtain the result with $r'=r/2^{k+1}$.
\end{proof} 

By Lemma~\ref{radius Lemma}, we may assume that $\{r_{k}\}$ is chosen so that property ({\it iii}) holds.  Let $\psi \in C^{\infty}(\mathbb{R})$ be such that $\psi$ is $1$ on $[0,1/2]$, decreasing on $[1/2,1]$ and $0$ on $[1,\infty)$.  Define functions $\{u_{k}\}$ by
\begin{equation}\label{non collapse test function}
u_{k} = e^{C_{k}}\psi \left( r_{k}^{-1}d^{T}(x,p_{k})\right)
\end{equation}
where $C_{k}$ is chosen so that
\begin{equation*}
(4\pi)^{n} = r_{k}^{-2n}\int_{T_{k}}u_{k}^{2}d\mu \leq 
e^{2C_{k}}r_{k}^{-2n} Vol(T_{k})
\end{equation*}
By assumption $r_{k}^{-2n}Vol(T_{k}) \rightarrow 0$, and so we necessarily have $C_{k} \rightarrow \infty$.  It is easy to see that the function $u_{k}$ defined in~(\ref{non collapse test function}) is basic, and $u_{k} \in W^{1,2}_{B}$.
We now apply the same argument as Perelman, using instead the transverse entropy functional, see \cite{KleinerLott},\cite{Pman}, \cite{SesumTian}. Proposition~\ref{non-collapse 2} is proved.  Proposition~\ref{non-collapse 1} follows easily by rescaling.  It is standard to check that if $g^{T}(t)$ is a solution of the Sasaki-Ricci flow, then $\tilde{g}^{T}(s)=(1-s)g^{T}(t(s))$ is a solution of the unnormalized Sasaki-Ricci flow for $t(s)=-\ln(1-s)$.  It is easy to see that the curvature assumption in Proposition~\ref{non-collapse 1} implies the curvature assumption in Proposition~\ref{non-collapse 2} at radius $\sqrt{1-s}$; it follows that
\begin{equation}
\kappa(1-s)^{n}\leq Vol_{\tilde{g}(s)}(T_{\tilde{g}(s)}(p, \sqrt{1-s})) = (1-s)^{n}Vol_{g(t(s))}(T_{g(t(s))}(p,1)
\end{equation}
which shows that Proposition ~\ref{non-collapse 1} follows from Proposition~\ref{non-collapse 2}.
\end{proof}
\begin{remark}
One might hope that in Proposition~\ref{non-collapse 2} the exponent $2n$ could be replaced with the exponent $2n+1-q$ which is optimal in light of Theorem~\ref{Gray Thm}.  However, the scaling argument  above, along with the fact that the normalized Sasaki-Ricci flow preserves the volume of $S$, shows that this is not possible unless $q=1$
\end{remark}

We now employ these non-collapsing results to obtain uniform transverse diameter bounds along the Sasaki-Ricci flow, which in light of Lemma~\ref{bounds by Cd} will prove Theorem~\ref{Pman thm}.
 
\begin{prop}\label{uniform diameter bounds}
There is a uniform constant $C$ such that $diam^{T}(S,g^{T}(t)) \leq C$.
\end{prop}
The proof is essentially identical to Perelman's, using the same adaptations as before.  We argue by contradiction.  Assume that the diameters are unbounded in time.  Denote by $d^{T}_{t}(z) = d^{T}_{t}(x,z)$ where $u(x,t)= \min_{y\in S} u(y,t)$.  Proposition~\ref{uniform diameter bounds} is then a consequence of the following two lemmas.  
\begin{Lemma}\label{ST claim 3.2}
For every $\epsilon>0$,  we can find $T(k_{1},k_{2})$, $k_{1} < k_{2}$, such that if $diam^{T}(S, g(t))$ is sufficiently large, then
\begin{enumerate}
\item[({\it i})] $Vol(T(k_{1},k_{2})) <\epsilon$
\item[({\it ii})] $Vol(T(k_{1},k_{2})) \leq 2^{10n} Vol(T(k_{1}+2, T_{2}-2))$
\end{enumerate}
\end{Lemma}
The proof of the the lemma is identical to the proof given in \cite{SesumTian}, and thus we omit it.  We note, that the non-collapsing theorem is crucial to the proof.

\begin{Lemma}\label{ST Lemma 3.3}
Let $k_{1}, k_{2}$ be as in Lemma~\ref{ST claim 3.2}.  Then there exists $r_{1}, r_{2}$ and a constant $C$ independent of time such that $2^{k_{1}}\leq r_{1} \leq 2^{k_{1}+1}$, $2^{k_{2}-1} \leq r_{2} \leq 2^{k_{2}}$, and
\begin{equation*}
\int_{T(r_{1}, r_{2})} R^{T} \leq C Vol(T(k_{1}, k_{2})) = CV
\end{equation*}
\end{Lemma}

\begin{proof}
By the coarea formula we have
\begin{equation*}
Vol(T(r)) = \int_{\{0\leq d^{T}\leq r\}} 1 d\mu \geq \int_{\{0\leq d^{T}\leq r\}} |\nabla d^{T}| d\mu = \int_{0}^{r} \mathcal{H}^{2n}\{d^{T}(x)=t\} dt.
\end{equation*}
In analogy with the Euclidean case, denote $S(t) = \mathcal{H}^{2n}\{d^{T}(x)=t\}$. 
\begin{step}
$\exists r_{1} \in [2^{k_{1}}, 2^{k_{1}+1}]$ such that $$S(r_{1}) \leq  2 \frac{V}{2^{k_{1}}}$$
\end{step}
If not, then
\begin{equation*}
Vol(T(k_{1}, k_{1}+1) \geq \int_{2^{k_{1}}}^{2^{k_{1}+1}}S(t)dt > 2V = 2Vol(T(k_{1}, k_{2}))
\end{equation*}
which is not possible as $k_{2} >>k_{1}$.  In a similar fashion we obtain
\begin{step}
$\exists r_{2} \in [2^{k_{2}-1}, 2^{k_{2}}]$ such that $$S(r_{2}) \leq  2 \frac{V}{2^{k_{2}}}$$
\end{step}
Now we have
\begin{equation*}
\begin{aligned}
\int_{T(r_{1},r_{2})} R^{T} &= \int_{T(r_{1},r_{2})} (R^{T}-n) + nVol(T(r_{1},r_{2})) \\&= -\int_{T(r_{1},r_{2})} \Delta u+ nVol(T(r_{1},r_{2})) \\
&\leq \int_{\{d^{T} = r_{1}\}} |\nabla u| + \int_{\{d^{T} = r_{2}\}} |\nabla u| + nV \\
&\leq C'(2^{k_{1} + 1}+1)2 \frac{V}{2^{k_{1}}} + C'  (2^{k_{2}} + 1) 2 \frac{V}{2^{k_{2}}} +nV \leq C V
\end{aligned}
\end{equation*}
where we have used the estimates in Lemma~\ref{bounds by Cd} in the last line.  Clearly $C$ is independent of $t$.

\end{proof}
\begin{proof}[Proof of Proposition~\ref{uniform diameter bounds}]
Suppose that the diameter is not uniformly bounded. That is, there exists a sequence of times $t_{i} \rightarrow \infty$ such that $diam^{T}(S, g(t_{i}))\rightarrow \infty$.  Let $\epsilon_{i}$ be a sequence of positive real numbers with $\epsilon_{i}\rightarrow 0$.  By Lemmas~\ref{ST claim 3.2} and~\ref{ST Lemma 3.3}, we can find sequences $k_{1}^{i}$, and $k_{2}^{i}$, such that
\begin{equation*}
Vol_{t_{i}}(T_{t_{i}}(k_{1}^{i},k_{2}^{i})) < \epsilon_{i}
\end{equation*}
\begin{equation*}
Vol_{t_{i}}(T_{t_{i}}(k_{1}^{i},k_{2}^{i})) \leq 2^{10n} Vol_{t_{i}}(T_{t_{i}}(k_{1}^{i}+2,k_{2}^{i}-2))
\end{equation*}
For each $i$, find $r_{1}^{i}$ and $r_{2}^{i}$ as in Lemma~\ref{ST Lemma 3.3}.  Let $\phi_{i}$ be a sequence of cutoff functions such that $\phi(z) =1$ on $[2^{k_{1}^{i}+2},2^{k_{2}^{i}-2}]$, and $\phi =0$ on $(-\infty, r_{1}^{i}] \cup [r_{2}^{i}, \infty)$.  Let $u_{i} = e^{C_{i}} \phi_{i}(d^{T}_{t_{i}}(x,p_{i}))$ such that $\int_{S}u_{i}^{2} =(4\pi)^{n}$. We have
\begin{equation*}
(4\pi)^{n} = e^{2C_{i}} \int_{S}\phi_{i}^{2} \leq e^{2C_{i}}\epsilon_{i}.
\end{equation*}
Since $\epsilon_{i} \rightarrow 0$, we have $C_{i} \rightarrow \infty$.  We plug the functions $u_{i}$ into the $\mathcal{W}^{T}$ functional, and follow the proof in \cite{SesumTian}.
\end{proof}

Combining Lemma~\ref{bounds by Cd}, and Proposition~\ref{uniform diameter bounds}, we have proved Theorem~\ref{Pman thm}.

\section{Appendix}

In the appendix we include the proofs of those facts requiring the use of transverse Sobolev spaces, which we felt were outside the theme of the main body of the paper.  In particular, we first prove:
\begin{Lemma}\label{poincare ineq}
Let $u$ satisfy the equation $g^{T}_{\bar{k}j}-Ric^{T}_{\bar{k}j} = \partial_{j}\partial_{\bar{k}}u$.  Then the following inequality
\begin{equation}
\frac{1}{Vol(S)}\int_{S}f^{2}e^{-u}d\mu \leq \frac{1}{Vol(S)}\int_{S}|\nabla f|^{2}e^{-u}d\mu +\left(\frac{1}{Vol(S)}\int_{S}f e^{-u} d\mu\right)^{2}
\end{equation}
holds for all $f\in C^{\infty}_{B}(S)$.
\end{Lemma}

The proof is elementary, and well known in the K\"ahler case.  We include a detailed proof, as it highlights some of the technical difficulties of working transversely.  We must study the operator
\begin{equation*}
Lf = -(g^{T})^{j\bar{k}}\nabla_{j}\nabla_{\bar{k}}f + (g^{T})^{j\bar{k}}\nabla_{\bar{k}}f\nabla_{j}u
\end{equation*}
on $C^{\infty}_{B}(S, \mathbb{C})$, the space of smooth, basic, \emph{complex} valued functions, which is \emph{not} elliptic.  However, we will still be able to analyze this operator, by observing that $L$ \emph{is} the restriction of an elliptic operator on $C^{\infty}(S, \mathbb{C})$.  In order to prove the lemma, we need to show that $L$ has a complete family of eigenvalues, its kernel is precisely the constants and the lowest strictly positive eigenvalue of L is no smaller than 1.  To do this, we need to study the basic Sobolev spaces;

\begin{definition}\label{basic sobolev norm}
Define $H^{r}_{B}(S,\mathbb{C}) = W^{r,2}_{B}(S, \mathbb{C})$ to be the closure of $C^{\infty}_{B}(S, \mathbb{C})$ with respect to the Sobolev norm induced by the metric $g^{T}$.  For example,
\begin{equation}
\| \phi \|_{1, B} = \int_{S} \phi \bar{\phi} d\mu + \int_{S} (g^{T})^{k\bar{j}} \nabla_{k}\phi \overline{\nabla_{j}\phi}d\mu +  \int_{S}(g^{T})^{j\bar{k}}  \nabla_{\bar{k}}\phi \overline{\nabla_{\bar{j}}\phi}d\mu.
\end{equation}
\end{definition}
\begin{remark}
First, we remark that we can make the above definition for $H^{r}_{B}(S,\mathbb{R})$ in a similar manner. Notice that the Sobolev norm defined above, and the restriction of the standard Sobolev norm to $C^{\infty}_{B}(S,\mathbb{C})$ agree, and hence the basic Sobolev norm is just obtained by restricting the standard Sobolev norm to the closure of the smooth basic functions.  However, this is not the case when we consider the Sobolev spaces $H^{r}(\Omega^{p}_{B})$ of $p$-forms, and so we choose to distinguish between the two to avoid confusion.
\end{remark}
From now on, we suppress the symbol $\mathbb{C}$, and let it be understood that we are considering complex valued functions.  In order to show that $L$ has a complete spectrum, it suffices to prove the existence of a Green's function for $L$.  The existence of the Green's function follows from a standard argument, once we have Rellich's theorem, and an elliptic a priori estimate.  Rellich's theorem was proved for general foliations in \cite{KamTond}.  We have

\begin{Lemma}[\cite{KamTond} Proposition 4.5]\label{basic Rellich}
$\forall r\geq 0$ and $t>0$ the inclusion $H^{r+t}_{B}(S) \hookrightarrow H^{r}_{B}(S)$ is compact.
\end{Lemma}

We can obtain the required elliptic a priori estimate, and regularity theorems by observing that on basic functions, $-L$ agrees with the elliptic operator
\begin{equation*}
\tilde{L} = g^{\alpha \beta} \nabla_{\alpha}\nabla_{\beta} - g^{\alpha \beta} \nabla_{\alpha}u\nabla_{\beta}.
\end{equation*}
By standard elliptic theory, the operator $\tilde{L}$ has the elliptic a priori estimate.  In particular, there is a constant $C_{r}$ so that, for any $\phi \in H^{r+1}_{B}(S) \subset H^{r+1}(S)$ we have
\begin{equation*}
\| \phi \|_{r+2} \leq C_{r} \left( \| L\phi \|_{r} + \| \phi \|_{r+1} \right).
\end{equation*}
By the remark following Definition~\ref{basic sobolev norm} we obtain,

\begin{Lemma}\label{basic elliptic reg}
Let $\phi \in H^{r+1}_{B}$ has $L\phi \in H^{r}_{B}$.  Then $\phi \in H^{r+2}_{B}$, and
\begin{equation*}
\| \phi \|_{r+2, B} \leq C_{r} \left( \| L\phi \|_{r,B} + \| \phi \|_{r+1,B} \right).
\end{equation*}
\end{Lemma}
\begin{remark}
In fact, this holds more generally.  It is well known (see for example \cite{KamTond}) that the restriction of the Sobolev norm on $H^{r}(\Omega)$ to the subspace $H^{r}(\Omega_{B})$ of basic forms induces a norm equivalent to the basic Sobolev norm. In particular, Lemma~\ref{basic elliptic reg} holds for basic forms.
\end{remark} 
Observe now that in Definition~\ref{basic sobolev norm} we can replace the measure $d\mu$ with $e^{-u}d\mu$ to obtain a \emph{weighted} basic Sobolev space, which we denote by $\widehat{H^{r}_{B}}(S)$.  It is clear that the weighted basic Sobolev norm is equivalent to the unweighted norm, and hence Lemmas~\ref{basic Rellich} and~\ref{basic elliptic reg} hold for the weighted norm.  A corollary of Lemmas~\ref{basic Rellich} and~\ref{basic elliptic reg} is
\begin{corollary}
$Ker L \subset H^{2}_{B}(S)$ is finite dimensional, and consists of smooth, basic functions.
\end{corollary}
Note moreover, that the operator $L$ is self-adjoint with respect to the inner product in $\widehat{H^{0}_{B}}(S)$.  In particular, if $f$ is basic, $L$ satisfies
\begin{equation*}
\int_{S} |\nabla f|^{2}e^{-u}d\mu = \langle Lf, f \rangle.
\end{equation*}
We see immediately that if $f \in Ker L \subset C^{\infty}_{B}(S)$, then $f$ is constant, and that all the eigenvalues of $L$ are real.  Now, $L : \widehat{H^{2}_{B}} \rightarrow \widehat{L^{2}_{B}}$, is self-adjoint and so we can apply the usual argument to obtain the existence of the Green's function $G : \widehat{L^{2}_{B}} \rightarrow \widehat{L^{2}_{B}}$ which is compact and self-adjoint, hence has a spectral decomposition, and yields the spectrum of $L$.  It now suffices to prove that the smallest positive eigenvalue of $L$ is no smaller than $1$.  This is a straightforward computation.  Assume that
\begin{equation*}
\int_{S} fe^{-u}d\mu =0 \text{  }\text{ that is} f \perp Ker L \text{ in } \widehat{L^{2}_{B}}.
\end{equation*}
We use a Bochner type argument.  Differentiating the equation yields
\begin{equation*}
\begin{aligned}
\nabla^{T}_{\bar{l}}(Lf) = &-(g^{T})^{j\bar{k}}\nabla^{T}_{j}\nabla^{T}_{\bar{l}}\nabla^{T}_{\bar{k}}f + Ric^{T}_{\bar{l}s}(g^{T})^{s\bar{k}}\nabla^{T}_{\bar{k}}f +( g^{T})^{j\bar{k}}\nabla^{T}_{\bar{l}}\nabla^{T}_{j}u \nabla^{T}_{\bar{k}}f\\& + ( g^{T})^{j\bar{k}}\nabla^{T}_{j}u \nabla^{T}_{\bar{l}}\nabla^{T}_{\bar{k}}f.
\end{aligned}
\end{equation*}
We now multiply by $(g^{T})^{m\bar{l}}\nabla^{T}_{m}f$ and integrate with respect to the weighted measure $e^{-u}d\mu$.  The first term is
\begin{equation*}
-\int_{S}(g^{T})^{j\bar{k}}(g^{T})^{m\bar{l}}\nabla^{T}_{j}\nabla^{T}_{\bar{l}}\nabla^{T}_{\bar{k}}f\nabla^{T}_{m}f e^{-u}d\mu.
\end{equation*}
Since $S$ is Sasakian, we observe that $\nabla^{T}_{j}\nabla^{T}_{\bar{l}}\nabla^{T}_{\bar{k}}f = \partial_{j}\nabla^{T}_{\bar{l}}\nabla^{T}_{\bar{k}}f$ as the connection coefficients with mixed barred and unbarred indices are zero.  Let $V^{\bar{l}}= (g^{T})^{m\bar{l}}\nabla^{T}_{m}f$.  Integration by parts yields
\begin{equation*}
\begin{aligned}
\int_{S}\nabla^{T}_{\bar{l}}\nabla^{T}_{\bar{k}}f  \big[&-(g^{T})^{j\bar{p}}\partial_{j}g^{T}_{\bar{p}q}(g^{T})^{q\bar{k}} V^{\bar{l}}\\
&+(g^{T})^{j\bar{k}}\partial_{j}V^{\bar{l}} - (g^{T})^{j\bar{k}}V^{\bar{l}}\partial_{j}u\\
&+(g^{T})^{j\bar{k}}V^{\bar{l}} g^{\alpha \beta}\partial_{j}g_{\beta \alpha} \big] e^{-u}d\mu.
\end{aligned}
\end{equation*}
By computing in preferred local coordinates, we see that $g^{\alpha \beta}\partial_{j}g_{\beta \alpha} = (g^{T})^{q\bar{p}}\partial_{j}g^{T}_{\bar{p}q}$.  Now, the transverse K\"ahler condition implies that the first and last terms cancel.  Using that $\nabla^{T}_{\bar{l}}\nabla^{T}_{j}u = \partial_{\bar{l}}\partial_{j} u = g^{T}_{\bar{l}j} - Ric^{T}_{\bar{l}j}$ we obtain
\begin{equation}\label{pos eigenvalue}
\begin{aligned}
\int_{S}\nabla_{\bar{l}}(Lf)(g^{T})^{m\bar{l}}\nabla_{m}f e^{-u}d\mu = &\int_{S}\big[(g^{T})^{j\bar{k}}(g^{T})^{m\bar{l}}\nabla^{T}_{\bar{l}}\nabla^{T}_{\bar{k}}f \nabla^{T}_{j}\nabla^{T}_{m}f\\ &+ (g^{T})^{m\bar{l}}\nabla^{T}_{\bar{l}}f \nabla^{T}_{m}f\big] e^{-u}d\mu.
\end{aligned}
\end{equation}
Since the first term in~(\ref{pos eigenvalue}) is positive, we see immediately that if $f$ is an eigenfunction of $L$ with eigenvalue $\lambda >0$, then $\lambda \geq 1$. We have succeeded in proving Lemma~\ref{poincare ineq}.

We now turn our attention to the proof of Lemma~\ref{minimizer}.  The method of proof is essentially that of Rothaus \cite{Roth}.  We refer the reader to \cite{EmLaNave} for a detailed argument in the K\"ahler case.  We first need the logarithmic Sobolev inequality.
\begin{prop}[logarithmic Sobolev inequality]
Let $(M^{n},g)$ be a closed, Riemannian manifold.  For any $a>0$ there exists a constant $C(a,g)$ such that if $\phi >0$ satisfies $\int_{M}\phi^{2} =1$, then
\begin{equation*}
\int_{M}\phi^{2}\log\phi d\mu \leq a\int_{M} |\nabla \phi|^{2}d\mu + C(a,g).
\end{equation*}
\end{prop}

\begin{proof}[Proof of Lemma~\ref{minimizer}]
Note that it suffices to show that we can find a minimizer for $\mathcal{W}^{T}(g,f,1)$.  Setting $w=(4\pi)^{n/2}e^{-f/2}$, we seek to minimize
\begin{equation*}
\int_{S}4|\nabla w|^{2} + \left(R^{T}-2\log w -n\log(4\pi) -2n\right)w^{2}d\mu.
\end{equation*}

The logarithmic Sobolev inequality implies any minimizing sequence is uniformly bounded in $W^{1,2}_{B}(S)$.  The same arguments as in \cite{Roth, EmLaNave} yield the existence of a non-negative minimizer $w_{1} \in W^{1,2}_{B}(S)$.  The minimizer must be a weak solution of the Euler-Lagrange equation
\begin{equation*}
-4\Delta w_{1} + R^{T}w_{1} -2w_{1}\log w_{1} -(n\log(4\pi) +2n)w_{1} = \mu^{T}(g,1)w_{1}.
\end{equation*}
By elliptic regularity we see that $w_{1} \in C^{\infty}$, and an application of the comparison principle shows that $w_{1} >0$.  Now, it is clear the $w_{1}$ is basic, as $w_{1} \in C^{\infty}(S) \cap W^{1,2}_{B}(S)$.
\end{proof}

\end{document}